\documentclass[12pt,reqno]{amsart}
\allowdisplaybreaks
\usepackage{graphicx}
\usepackage{amssymb}
\usepackage{bm}

\newtheorem{theorem}{Theorem}
\newtheorem{lemma}[theorem]{Lemma}
\newtheorem{proposition}[theorem]{Proposition}
\newtheorem{corollary}[theorem]{Corollary}

\theoremstyle{definition}
\newtheorem{definition}[theorem]{Definition}

\newtheorem{notation}[theorem]{Notation}

\theoremstyle{remark}
\newtheorem{remark}[theorem]{Remark}

\newcommand{\alg}{\operatorname{alg}}

\newcommand{\centre}{\mathaccent"7017}
\newcommand{\cU}{\mathcal{U}}
\newcommand{\cV}{\mathcal{V}}
\newcommand{\cB}{\mathcal{B}}
\newcommand{\cE}{\mathcal{E}}

\newcommand{\cov}{\operatorname{cov}}
\newcommand{\rk}{k}
\newcommand{\rC}{\mathrm{C}\kern1pt}
\newcommand{\rO}{\operatorname{O}}
\newcommand{\lo}{\operatorname{o}}
\newcommand{\tr}{\operatorname{tr}}
\newcommand{\E}{\operatorname{E}}

\newcommand{\bC}{\mathbb{C}}

\newcommand{\bZ}{\mathbb{Z}}

\newcommand{\peq}{{p \mskip 2mu\cdot_\epsilon q}}
\newcommand{\pep}{{p \mskip 2mu\cdot_\epsilon p}}

\let\phi=\varphi

\newcommand{\spoke}[1]{\mathit{Sp}^+(#1)}
\newcommand{\rspoke}[1]{\mathit{Sp}^-(#1)}

\newcommand{\Tr}{\operatorname{Tr}}

\newcommand{\Wg}{\operatorname{Wg}}

\def\ff{\varphi}

\newcommand{\cA}{\mathcal{A}}
\newcommand{\ds}{\displaystyle}
\newcommand{\cP}{\mathcal{P}}
\newcommand{\sgn}{\textrm{sgn}}

\theoremstyle{definition}

\newcommand{\thebottomline}{%
\renewcommand{\thefootnote}{}
\renewcommand{\footnoterule}{}
\phantom{M}\footnote{\tiny \noindent\today}}

\newcommand{\ab}{\allowbreak}

\makeatletter
\newcommand{\myleqno}{\let\veqno\@@leqno}
\newcommand{\myreqno}{\let\veqno\@@reqno}
\makeatother

\title[Real Second Order Freeness]{Real Second Order Freeness\\ 
       and Haar Orthogonal Matrices}

\author[mingo]{James A. Mingo$^{(*)}$}
\address{Department of Mathematics and Statistics, Queen's
  University, Jeffery Hall, Kingston, Ontario, K7L 3N6,
  Canada} 
\email{mingo@mast.queensu.ca}
\thanks{$^*$ Research supported by a Discovery Grant from
  the Natural Sciences and Engineering Research Council of
  Canada}

\author[popa]{Mihai Popa$^{(*)(\ddagger)}$}
\address{Department of Mathematics and Statistics, Queen's
  University, Jeffery Hall, Kingston, Ontario, K7L 3N6,
  Canada, and \newline
${}\hspace{.4cm}{}$ Institute of Mathematics ``Simion Stoilow'' of the Romanian Academy, P.O. Box 1-764, Bucharest, RO-70700, Romania }

\thanks{$(\ddagger)$ Research supported by the Natural Science Foundation of China Grant No. 11150110456, and the Romanian National Authority for Scientific Research, CNCS  UEFISCDI, Project Number PN-II-ID-PCE-2011-3-0119}
\email{popa@mast.queensu.ca}

\begin{document}

\begin{abstract} 
We demonstrate the asymptotic real second order freeness of
Haar distributed orthogonal matrices and an independent
ensemble of random matrices.  Our main result states that if
we have two independent ensembles of random matrices with a
real second order limit distribution and one of them is
invariant under conjugation by an orthogonal matrix, then
the two ensembles are asymptotically real second order
free. This captures the known examples of asymptotic real
second order freeness introduced by Redelmeier \cite{r1,r2}.
\end{abstract}

\maketitle


\section{Introduction}

The large $N$ behaviour of random matrices has been actively
studied since Wigner's celebrated semi-circle law was found
in 1955, \cite{w}. Subsequently in 1967 Marchenko and Pastur
found the limit distribution for Wishart matrices
\cite{mp1}, now called the Marchenko-Pastur distribution.
The essential point of these discoveries is that for many
ensembles of random matrices the description of the
distribution of the eigenvalues gets much simpler in the
large $N$ limit. Much subsequent work has been devoted to
expanding and refining this work, see for example the recent
book of Anderson, Guionnet, and Zeitouni \cite{agz}.

Another direction of research in random matrices deals with
the interaction of independent ensembles of random
matrices. In this direction one studies the limit eigenvalue
distribution of sums and products of ensembles whose limit
distributions are already known. The direction was
discovered by Voiculescu in his work on free probability. In
\cite{v1} and later in \cite{v2}, Voiculescu showed that
independent ensembles were asymptotically free if at least
one was unitarily invariant.  Recall that if two random
variables are freely independent then there is a universal
rule for finding the mixed moments from the moments of the
individual random variables. One does this either
analytically by using the $R$ and $S$ transform, see
\cite{vdn}, or combinatorially using free cumulants, see
\cite{ns}.

In the last two decades the fluctuations of the eigenvalues
have been studied both in the physics and the mathematics
literate, see e.g. \cite{az, bs, fmp, j, k, kkp}. In
\cite{mn} it was shown that the fluctuations of Wishart
matrices could be analyzed using the non-crossing diagrams
introduced in \cite{s}, but by using an annulus instead of a
disc or line, see Figure
\ref{figure:noncrossingpermutations}, hence all the
combinatorial techniques developed by Nica and Speicher
\cite{ns} could be brought to bear on the study of
fluctuations. Thus motivated, second order freeness was
introduced in \cite{ms, mss} and later higher order freeness
in \cite{cmss}.

The point of second and higher order freeness is that it
enables one to do for fluctuation moment and higher order
trace-moments what Voiculesu's first order freeness did for
moments. In particular if two random variables are second
order free and one knows the moments and the fluctuation
moments of each variable then there is a universal rule for
finding fluctuation moments of sums and products, see
\cite{mst}.

In \cite{cmss, mn, ms, mss} the random matrices considered
were either Hermitian or unitary. This left the question of
how to deal with real symmetric and orthogonal matrices. On
the first order level the techniques of Voiculescu were
equally applicable to real and complex ensembles. However it
was shown in \cite{r1,r2} that the universal rule found in
\cite{ms} needed to be modified for the real case; in
particular the transpose of the various operators made an
appearance. This led to a new kind of second order freeness,
called real second order freeness in \cite{r1,r2}.

The non-crossing diagrams introduced in \cite{mn} had to
augmented by diagrams in which the orientation of one of the
circles was reversed. The operators on the reversed side get
transposed. One can give a heuristic interpretation of this
using maps on surfaces, see \cite{lz}. In the complex case
we only work with orientable surfaces and in the real case
we also have to also deal with non-orientable surfaces. So
we imagine that our surfaces are marked our operators and
the graphs tell us how they get multiplied, see Figure
\ref{figure:paritypreserved}. Wherever we put an operator on
the front side of the surface, we put its transpose on the
back. The non-orientability of the surface means that we can
cross from font to back and see the transposed operators,
something that we cannot do in the complex case.

The main result of this paper, Theorem
\ref{theorem:mastertheorem}, asserts that if $\{A_i\}_i$ and
$\{B_j\}_j$ are independent ensembles of random matrices and
if at least one of them is invariant under conjugation by an
orthogonal matrix then the ensembles are asymptotically real
second order free. The proof of this theorem occupies nearly
the whole paper. This theorem is the orthogonal version of a
theorem in \cite{mss}, where we assumed that one of the
ensembles is invariant under conjugation by an unitary
matrix. While the statements of the two theorems are similar
the proofs follow quite different paths. In \cite{mss} the
asymptotics of the cumulants of the unitary Weingarten
function, from \cite{c}, were heavily used. In this paper we
only need the multiplicitivity of the leading order of the
orthogonal Weingarten function, see \cite{cs}. We work with
centred elements and this obviates the need to work with the
cumulants of the Weingarten function.

\subsubsection*{Illustrative examples}

Let us conclude this introduction with some
examples. Suppose that $A_1, A_2, A_3, A_4$ are $d \times d$
deterministic matrices and $O$ is a $d \times d$ Haar
distributed random orthogonal matrix and $U$ be a $d \times
d$ Haar distributed random unitary matrix.  From
\cite[Prop. 3.4]{mss} we have
\[
\E(\Tr(UA_1U^{-1}A_2)) =
d^{-1}\Tr(A_1) \Tr(A_2)\ \mbox{and}
\]
\[
\E(\Tr(UA_1UA_2)) = 0.
\]
According to Proposition \ref{proposition:exact_formula}
\[
\E(\Tr(OA_1O^{-1}A_2))
= d^{-1}\Tr(A_1)\Tr(A_2) \ \mbox{and}
\]
\[
\E(\Tr(OA_1OA_2)) =
d^{-1} \Tr(A_1A_2^t).
\]
So we already see a bit a difference between the orthogonal
and unitary cases; namely the appearance of transposes in
lower order terms. When we consider covariances we see more
differences. First in the unitary case we have
\begin{multline*}
\cov(\Tr(UA_1U^{-1}A_2), \Tr(UA_3U^{-1}A_4)) \\
=
\frac{d^{-4}}{1-d^{-2}} \Tr(A_1)\Tr(A_2)\Tr(A_3)\Tr(A_4) \\
+
\frac{d^{-2}}{1-d^{-2}} \Tr(A_1A_3) \Tr(A_2A_4) \\
-
\frac{d^{-3}}{1-d^{-2}} \left\{
\Tr(A_1A_3)\Tr(A_2)\Tr(A_4) + \Tr(A_1)\Tr(A_2A_4) \Tr(A_3)
\right\}.
\end{multline*}

Now in the orthogonal case we have
\begin{multline*}
(1 + d^{-1}\!\!-2d^{-2})\cov(\Tr(OA_1O^{-1}A_2), \Tr(OA_3O^{-1}A_4))\\
=
d^{-4} \{\Tr(A_1) \Tr(A_2) \Tr(A_3) \Tr(A_4)\\
+
\Tr(A_1) \Tr(A_2) \Tr(A_3^t) \Tr(A_4^t)\}\\
- d^{-3}\{
\Tr(A_1A_3)\Tr(A_2) \Tr(A_4) + \Tr(A_1A_3^t)\Tr(A_2) \Tr(A_4^t)\\
+ 
\Tr(A_1)\Tr(A_2A_4)\Tr(A_3) + \Tr(A_1)\Tr(A_2A_4^t)\Tr(A_3^t)\}\\
+ (d^{-2}+ d^{-3})\{
\Tr(A_1A_3)\Tr(A_2A_4) + \Tr(A_1A_3^t)\Tr(A_2A_4^t)\} \\
- d^{-3}\{ \Tr(A_1A_3^t)\Tr(A_2A_4)  + \Tr(A_1A_3)\Tr(A_2A_4^t)\}.
\end{multline*}

Note the similarity to the unitary case except that each
term of leading order appears twice--once with no transposes
and once with transposes on $A_3$ and $A_4$. Moreover when
the $A_i$'s are centred, i.e. $\Tr(A_i) = 0$, the only
remaining terms are $\Tr(A_1A_3)\Tr(A_2A_4)$ and
$\Tr(A_1A_3^t)\Tr(A_2A_4^t)$. These terms correspond to
spoke diagrams which are discussed in the next section, see
Figure \ref{figure:bothspokediagrams}. By working with
centred elements the number of terms is significantly
reduced, it is in this way that we can skip the calculations
requiring the cumulants of the Weingarten function.

\subsubsection*{The Organization of the Paper} 
In section \ref{section:noncrossingdiagrams} we review the
definitions of non-crossing partitions. In section
\ref{section:traceofproduct} we use the Weingarten function
of \cite{cs} to compute the trace of a product of orthogonal
matrices and independent random matrices. This is how the
calculations in the examples above were done. In section
\ref{section:spokediagrams} we prove two important lemmas on
a special kind of non-crossing partition called a spoke
diagram. These are the only diagrams that survive in the
large $d$ limit. In section
\ref{section:realsecondorderfreeness} we recall the notions
of second order freeness from \cite{r1,r2} and prove that
real second order freeness satisfies an associative law. In
section \ref{section:firstorderfreenes} we prove that Haar
distributed orthogonal matrices and an independent ensemble
are first order free. That this could be done was already
suggested by Voiculescu in \cite{v1} some twenty years ago
and was later proved in \cite[Thm.~5.2]{cs}. In section
\ref{section:fluctuationmoments} we show that the
fluctuation moments of Haar distributed orthogonal matrices
and an independent ensemble of random matrices satisfy the
universal rule required for second order freeness. In
section \ref{section:vanishing} we show that the third and
higher cumulants of traces of products of Haar distributed
orthogonal matrices and an independent ensemble of random
matrices satisfy the final condition for asymptotic real
second order freeness. This completes the proof of their
asymptotic real second order freeness. In section
\ref{section:mainresults} we use this result to obtain all
our other results on asymptotic real second order freeness.
In section \ref{section:finalremarks} we present some
concluding remarks and indications of future work.

\section{Non-crossing diagrams and pairings}
\label{section:noncrossingdiagrams}

Central to the combinatorial approach to freeness is the
idea of a non-crossing partition. A partition of $[n]$ is
non-crossing is one in which the blocks can be drawn in a
non-crossing way; see the left half of Figure
\ref{figure:noncrossingpermutations}. For second order
freeness we need non-crossing annular partitions. This means
we can draw the blocks on an annulus in a non-crossing way;
see the right half of Figure
\ref{figure:noncrossingpermutations}. In the case of second
order freeness additional information about the partitions
is needed, namely the order in which they visit the
points. For this reason we regard our partitions as
permutations by interpreting the blocks of the partition as
cycles in the cycle decomposition of the corresponding
permutation.

\begin{notation}
For any integer $n \geq 1$, let $[n] = \{1, 2, 3, \dots,
n\}$.  Let $\cP(n)$ be the set of all partitions of
$[n]$. For any partition $\pi$ of $[n]$ let $\#(\pi)$ denote
the number of blocks of $\pi$, and $|\pi| = n -
\#(\pi)$. The set $\cP(n)$ is a partially ordered set in
which $\pi \leq \sigma$ means every block of $\pi$ is
contained in some block of $\sigma$. With this order
$\cP(n)$ is partially ordered set and is in fact a
lattice. We denoted the join of two partitions $\pi$ and
$\sigma$ by $\pi \vee \sigma$.
\end{notation}

Given a permutation it can be difficult to decide if there
is a non-crossing way of drawing its cycles, however there
is an algebraic way to see if such a diagram exists. Let
$\gamma = (1, \dots, m)(m+1, \dots, m+n)$ and let $\pi$ be a
permutation of $[m+n]$ and denote by $\langle \pi, \gamma
\rangle$ the subgroup of $S_n$ generated by $\pi$ and
$\gamma$. If the subgroup $\langle \pi, \gamma \rangle$ acts
transitively on $[m+n]$ then we have that $\pi$ is
non-crossing if and only if
\begin{equation}\label{equation:firstgeodesic}
\#(\pi) + \#(\pi^{-1}\gamma) = m + n.
\end{equation}
Note that the condition that $\langle \pi, \gamma \rangle$
act transitively is the same as requiring that there is at
least one cycle of $\pi$ that contains points in both cycles
of $\gamma$. When this happen we shall say that $\pi$
\textit{connects} the cycles of $\gamma$

\begin{figure}
\hfill\includegraphics{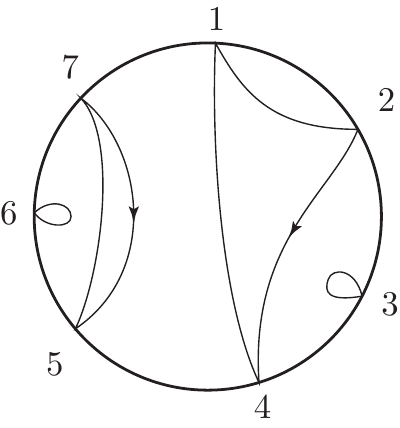}\hfill\includegraphics{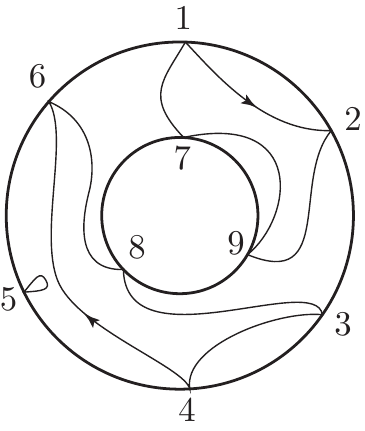}\hfill\hbox{}
\caption{\label{figure:noncrossingpermutations} On the left
  we have the non-crossing disc permutation $(1, 2, 4)(3)(5,
  7)(6)$. On the right we have the non-crossing annular
  permutation $(1, 2, 9, 7)\ab(3, 4, 6, 8)\ab(5)$.}
\end{figure}

We can extend this to the case of $\gamma$ having any number
of cycles. Let $\pi$ and $\gamma$ be permutations of
$[n]$. Let $k$ be the number of orbits of $\langle \pi,
\gamma \rangle$. Then
\begin{equation}\label{equation:genusexpansion}
\#(\pi) + \#(\pi^{-1}\gamma) + \#(\gamma) \leq n + 2k
\end{equation}
with equality only if $\pi$ is non-crossing with respect to
$\gamma$, see e.g. \cite[Remark 2.11]{mn}.

In the case of real second order freeness we require an
additional set of non-crossing diagrams, we call these
\textit{reversed non-crossing annular permutations}. If we
let $\gamma' = (1,\dots,m)(m+n, m+n-1, \dots, m+2, m+1)$
then we say that a permutation $\pi \in S_{m+n}$ is a
reversed non-crossing permutation of a $(m,n)$-annulus if
\[
\#(\pi) + \#(\pi^{-1}\gamma') = m + n.
\]
Notice that this is the same condition as in Equation
(\ref{equation:firstgeodesic}) but $\gamma$ is replaced with
$\gamma'$. Graphically, this corresponds to using the same
orientation for labelling the points on each circle; see the
right hand side of Figure \ref{figure:bothspokediagrams}.

A special kind of a non-crossing annular permutation that we
shall make use of is that of a \textit{spoke diagram}, see
Figure \ref{figure:bothspokediagrams}.  Recall that a
\textit{pairing} of $[n]$ is a partition in which each block
has two elements. We usually regard a pairing as a
permutation, by considering each block to be a cycle with
two elements.  By a \textit{standard spoke diagram} we mean
a non-crossing pairing of an $(m, n)$-annulus in which all
pairs connect the two circles. Note that means that $m = n$
and there is $l$ such that $m+1 \leq l \leq 2m$ such that
every cycle of $p$ is of the form $(k, \gamma^{-k}(l))$ for
$1 \leq k \leq m$.

By a \textit{reversed spoke diagram} we mean a reversed
non-crossing annular pairing in which all blocks connect the
two circles; see Figure \ref{figure:bothspokediagrams}. By a
\textit{spoke diagram} we mean either a standard or reversed
spoke diagram. See Figure
\ref{figure:bothspokediagrams}. Note that means that $m = n$
and there is $l$ such that $m+1 \leq l \leq 2m$ such that
every cycle of $p$ is of the form $(k, \gamma^{k}(l))$ for
$1 \leq k \leq m$.

\begin{figure}[t]
\hfill\includegraphics{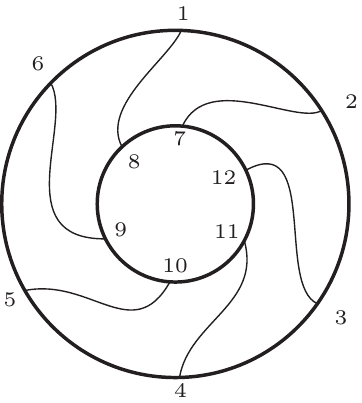}\hfill\includegraphics{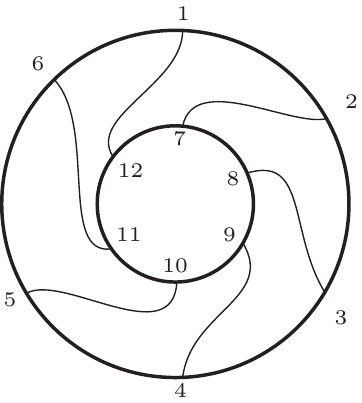}\hfill\hbox{}
\caption{\label{figure:bothspokediagrams} On the left we
  have a non-crossing pairing of a $(6, 6)$-annulus in which
  all blocks connect the two circles, i.e. a standard spoke
  diagram. Note that the two circles have opposite
  orientations. In the figure on the right we have a
  reversed non-crossing pairing of a $(6,6)$-annulus. i.e. a
  reversed spoke diagram. Note that the two circles having
  the same orientation.  }
\end{figure}

We denote by $\cP_2(n)$ the pairings of $[n]$. If $p$ is a
pairing of $[n]$ and $(r, s)$ is a cycle of $p$ we shall
denote this by $(r, s) \in p$. We denote by
$\spoke{m}$ the set of all standard spoke
diagrams and by $\rspoke{m}$ the set of all
reversed spoke diagrams.

Given a permutation $\pi \in S_n$, we shall frequently
consider the cycles of $\pi$ as a partition of $[n]$. This
map $S_n \longrightarrow \cP(n)$ forgets the order of
elements in a cycle and so is not a bijection. Conversely
given a partition $\pi \in \cP(n)$ we put the elements of
each block into increasing order and consider this a
permutation.  Restricted to pairings this is a bijection.


\section{The Trace of a Product}
\label{section:traceofproduct}

Given a permutation $\sigma \in S_n$ and $d \times d$
matrices $A_1, \dots, A_n$ we let $a^{(i)}_{p,q}$ be the
$(p,q)$-entry of $A_i$ and
\begin{equation}\label{equation:permutation}
\Tr_\sigma(A_1, \dots, A_n) = \sum_{i_1, \dots, i_n = 1}^d
a^{(1)}_{i_1 i_{\sigma(1)}} \cdots
a^{(n)}_{i_ni_{\sigma(n)}}.
\end{equation}
This expression can also be written as a product of traces
as follows. Write $\sigma = c_1 \cdots c_k$ in cycle
form. If $c = (i_1, \dots, i_r)$ is a cycle of $\sigma$ we
let $\Tr_c(A_1, \dots, A_n) = \Tr(A_{i_1} \cdots
A_{i_r})$. Then
\[
\Tr_\sigma(A_1, \dots, A_n) = \prod_{i=1}^k \Tr_{c_i}(A_1,
\dots, A_n).
\]

Let $O = (o_{ij})$ be a $d \times d$ Haar distributed random
orthogonal matrix and $\{Y_1, \dots , Y_n\}$ be $d \times d$
random matrices whose entries have moments of all
orders. Let $\gamma \in S_n$ be a permutation, and let
$\epsilon_1, \epsilon_2, \dots,\ab \epsilon_n \in \{-1,
1\}$. In this section we wish to find a simple expression
for
\[
\E(\Tr_\gamma(O^{\epsilon_1}Y_1, \dots, O^{\epsilon_n}Y_n)).
\]

We shall use the Weingarten function introduced by Collins
and \'Sniady \cite{cs}. The Weingarten expresses the
expectation $\E(o_{i_1i_{-1}} \cdots\ab o_{i_ni_{-n}})$ as a
sum over pairings of $[n]$.  The first question we need to
address is, for two pairings $p$ and $q$, the relationship
between the cycles of $pq$ and the blocks of $p \vee q$. See
Figure \ref{figure:pairingloops}. This is a standard fact;
for the reader's convenience and to establish our notation
we give a proof.

\begin{lemma}\label{lemma:pairingproduct}
Let $p, q \in \cP_2(n)$ be pairings and $(i_1, i_2, \dots,
i_k)$ a cycle of $pq$. Let $j_r = q(i_r)$. Then $(j_k,
j_{k-1}, \dots , j_1)$ is also a cycle of $pq$, and these
two cycles are distinct; $\{i_, \dots, i_k, j_1,\dots,
j_k\}$ is a block of $p \vee q$ and all are of this form;
$2\#(p \vee q) = \#(pq)$.
\end{lemma}

\begin{proof}
We have $pq(i_r)= i_{r+1}$, thus $j_r = q(i_r) =
p(i_{r+1})$. Hence $pq(j_{r+1}) = p(q(q(i_{r+1}))) =
p(i_{r+1}) = j_r$. If $\{i_1, \dots, i_k\}$ and $\{j_1,
\dots, j_k\}$ were to have a non-empty intersection then,
for some $n$, $q(pq)^n$ would have a fixed point, but this
would in turn imply that either $p$ or $q$ had a fixed
point, which is impossible. Since $\{q(i_r) \}_r = \{
j_s\}_s$ and $\{p(j_s)\}_s = \{i_r\}_r$, $\{i_, \dots, i_k,
j_1,\dots, j_k\}$ must be a block of $p \vee q$.  Since
every point of $[n]$ is in some cycle of $pq$, all blocks
must be of this form. Since every block of $p \vee q$ is the
union of two cycles of $pq$, we have $2\#(p \vee q) =
\#(pq)$.
\end{proof}

\begin{notation}\label{notation:permutationextension}
Let $[-n] = \{-n, -n+1, \dots, -2, -1\}$ and $[\pm n] = [-n]
\cup [n]$. Let $\delta$ be the permutation of $[\pm n]$
which sends $k$ to $-k$ for $k \in [\pm n]$. Since each
cycle of $\delta$ is of the form $(k, -k)$, we shall also
regard $\delta$ as a pairing of $[\pm n]$. If $\epsilon \in
\bZ_2^n = \{-1, 1\}^n$, let $\delta_\epsilon$ denote the
permutation of $[\pm n]$ given by $k \mapsto \epsilon_{|k|}
k$.

Given $\pi$ a permutation on $[n]$ we shall regard $\pi$
also a permutation of $[\pm n]$ where for $1 \leq k \leq n$,
we let $\pi(-k) = -k$. Let $\gamma$ be the permutation
of $[n]$ with the one cycle $(1,2,3, \dots, n)$, but
following the convention mentioned above we also have
$\gamma(-k) = -k$ for $1 \leq k \leq n$.

\end{notation}

\begin{figure}\label{figure:pairingloops}
\includegraphics{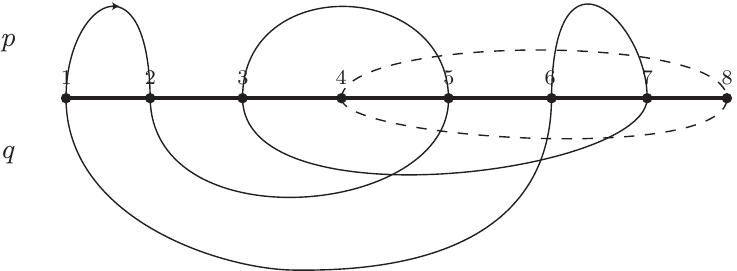}
\caption{In this example $n = 8$, $p =
  (1,2)\ab(3,5)\ab(4,8)\ab(6,7)$, and $q =
  (1,6)(2,5)(3,7)(4,8)$. Then $pq = (1,7,5)(2,3,6)(4)(8)$
  and $p \vee q = \{(1,2,3,5,6,7)(4,8)\}$.}
\end{figure}

\begin{lemma}\label{lemma:pairingloopsbis}
Let $p, q \in \cP_2(n)$ be pairings then $\#(pq) =
\#(p\delta q)$.

\end{lemma}

\begin{proof}
Note that for $1 \leq k \leq n$ we have $p\delta q(k) < 0$
and $p\delta q(-k) > 0$. Thus the elements in an orbit of
$p\delta q$ always alternate in sign. Moreover $(p \delta
q)^2 = pq$. Hence the positive elements of a cycle of $p
\delta q$ form a cycle of $pq$. Conversely let $(i_1, i_2,
\dots, i_r)$ be a cycle of $pq$. Then $(i_1, -q(i_1), i_2,
-q(i_2), \dots , i_r, -q(i_r))$ is a cycle of $p \delta
q$. This establishes a bijection between the cycles of $p
\delta q$ and the cycles of $pq$.
\end{proof}

The pairings of $[\pm n]$ shall be denoted $\cP_2(\pm
n)$. For a pairing $p \in \cP_2(\pm n)$, and a $2n$-tuple
$i=(i_1, i_{-1}, \dots, i_n, i_{-n})$ we write $i = i \circ
p$ to mean that whenever $p(r) = s$ we have $i_r = i_s$. For
a $d \times d$ matrix $A$ let $A^{(-1)} = A^t$, the
transpose of $A$, and $A^{(1)} = A$. For $\eta = (\eta_1,
\eta_2, \dots, \eta_n) \in \bZ_2^n$ and $\pi \in S_n$, let
$\Tr_{(\pi, \eta)}(A_1, \dots, A_n) =
\Tr_\pi(A^{(\eta_1)}_1, \dots, A^{(\eta_n)}_n)$

\begin{lemma}\label{lemma:pairingtopermutation}
Let $p \in \cP_2(\pm n)$. The there is $\pi \in S_n$ and
$\eta \in \bZ_2^n$ such that
\[
\mathop{\sum_{i_1, i_{-1}, \dots, i_{n}, i_{-n}=1}}_{i = i
  \circ p}^d a^{(1)}_{i_1i_{-1}} a^{(2)}_{i_2i_{-2}} \cdots
a^{(k)}_{i_{k}i_{-k}} \cdots a^{(n)}_{i_{n}i_{-n}} =
\Tr_\pi(A^{(\eta_1)}_1, \dots, A^{(\eta_n)}_n)
\]
\end{lemma}

\begin{proof}
We saw that the cycle decomposition of $p\delta$ may be
written $c_1 {c_1}' \cdots c_s {c_s}'$ where ${c_i}' =
\delta c_i^{-1} \delta$. It is arbitrary which of the pair
$\{ c_i, {c_i}'\}$ is called $c_i$ and which ${c_i}'$.

For each $i$, choose a representative of each pair $\{ c_i,
{c_i}'\}$, say $c_1, c_2, \dots,\ab c_s$. For each $i$ we
construct a cycle $\tilde c_i$ as follows. Suppose $c_i =
(l_1, \dots, l_r)$. Let $\tilde c_i = (j_1, j_2, \dots,
j_r)$ where
\[
j_k = 
\begin{cases}
-l_k & l_k < 0 \\
l_k & l_k > 0 
\end{cases}
\mbox{\ and\ } 
\eta_{j_k} =
\begin{cases}
-1 & l_k < 0 \\
1 & l_k >0
\end{cases}.
\]
Note that $j_k = \eta_{j_k} l_k = |l_k|$. Then we let $\pi =
\tilde c_1 \cdots \tilde c_s$ and $\eta = (\eta_1, \dots,
\eta_n)$.

We denote the $(m,n)$ entry of $A_i^{(\eta_i)}$ by
$a^{(i,\eta_i)}_{m,n}$. Let $(l_1, \dots, l_r)$ be a cycle
of $p\delta$. Let $(j_1, \dots, j_1)$ and $(\eta_1, \dots,
\eta_n)$ be as above i.e. $j_k = |l_k|$ and $\eta_{j_k} =
l_k/|l_k|$. Then
\[
a^{(j_k)}_{i_{j_k}\, i_{-j_k}} = \left.\begin{cases} \vrule
  width 0pt depth 15pt a^{(j_k, \eta_{j_k})}_{i_{j_k}\,
    i_{-j_k}} & \mbox{ if\ } \eta_{j_k} = 1 \\ a^{(j_k,
    \eta_{j_k})}_{i_{-j_k}\, i_{j_k}} & \mbox{ if\ }
  \eta_{j_k} = -1 \\
\end{cases}\right\}
= a^{(j_k, \eta_{j_k})}_{i_{l_k}\, i_{-l_k}}.
\]
Thus
\[
a^{(j_1)}_{i_{j_1}\, i_{-j_1}} \cdots
a^{(j_r)}_{i_{j_r}\, i_{-j_r}}
= a^{(j_1, \eta_{j_1})}_{i_{l_1}\, i_{-l_1}} \cdots 
a^{(j_r, \eta_{j_r})}_{i_{l_r}\, i_{-i_r}}
\]
Note that $i_{-l_k}= i(\delta(l_k)) = i(p\delta(l_k)) =
i(l_{k+1}) = i_{l_{k+1}}$, as $i = i \circ p$. Thus
\begin{eqnarray*}
&&
\mathop{\sum_{i_1, \dots, i_{-n}=1}}_{i = i \circ p}^d
a^{(1)}_{i_1i_{-1}} a^{(2)}_{i_2i_{-2}} \cdots a^{(k)}_{i_{k}i_{-k}}
\cdots a^{(n)}_{i_{n}i_{-n}}  \\
&=& 
\mathop{\sum_{i_1, \dots, i_{-n}=1}}_{i = i \circ p}^d
\mathop{\prod_{\tilde c \in \pi}}_{ \tilde c = (j_1,\dots, j_r)}
a^{(j_1)}_{i_{j_1}\, i_{-j_1}} \cdots
a^{(j_r)}_{i_{j_r}\, i_{-j_r}} \\
&=&
\mathop{\sum_{i_1, \dots, i_{2n}=1}}_{i = i \circ p}^d
\mathop{\prod_{\tilde c \in \pi}}_{ \tilde c = (j_1,\dots, j_r)}
a^{(j_1, \eta_{j_1})}_{i_{l_1}\, i_{-l_1}} \cdots 
a^{(j_r, \eta_{j_r})}_{i_{l_r}\, i_{-i_r}} \\
& = &
\mathop{\prod_{\tilde c \in \pi}}_{ \tilde c = (j_1,\dots, j_r)}
\Tr(A_{j_1}^{(\eta_{j_1})} \cdots A_{j_r}^{(\eta_{j_r})}) \\
& = &
\Tr_\pi(A_1^{(\eta_1)}, \dots, A_n^{(\eta_n)})
\end{eqnarray*}
\end{proof}

\begin{remark}
The pair $(\pi, \eta)$ constructed in Lemma
\ref{lemma:pairingtopermutation} is not unique; however
since
\[
\Tr(A_{j_1}^{(\eta_{j_1})} \cdots A_{j_r}^{(\eta_{j_r})})
=
\Tr(A_{j_r}^{(-\eta_{j_r})} \cdots A_{j_1}^{(-\eta_{j_1})})
\]
the value of $\Tr_\pi(A_1^{(\eta_1)}, \dots,
A_n^{(\eta_n)})$ is independent of the choices made.
\end{remark}

\begin{notation}
Let $\bC[\cP_2(n)]$ be the inner product vector space with
orthonormal basis $\cP_2(n)$. For an integer $d \geq n$,
define $\phi : \bC[\cP_2(n)] \longrightarrow \bC[\cP_2(n)]$
by
\[
\langle \phi(p), q \rangle = d^{\#(p \vee q)}
\]
\end{notation}

In \cite[\S 3]{cs}, Collins and \'Sniady showed that $\phi$
is an invertible linear transformation and denoted its
inverse $\Wg$, the \textit{orthogonal Weingarten
  function}. From the construction, $\langle \Wg(p),
q\rangle$ is always a rational function of $d$. Collins and
\'Sniady showed \cite[Thm. 3.13]{cs} that given $p, q \in
\cP_2(n)$ if we expand in power series in $d^{-1}$ then we
have
\begin{equation}\label{equation:weingartenorder}
\langle \Wg(p), q \rangle = \rO(d^{-n + \#( p \vee q)}).
\end{equation}

\begin{remark}\label{remark:leadingorder}
It was shown in \cite{cs} that the coefficient of $d^{-n +
  \#(p \vee q)}$ can be written as a product of signed
Catalan numbers. Indeed, write $pq = \rho q \rho^{-1} q$ and
factor $\rho$ into a product of cycles $c_1 \cdots c_k$. Let
$C_m$ be the $m^{th}$ Catalan number
$\frac{1}{m+1}\binom{2m}{m}$. Then the coefficient of $d^{-n
  + \#(p \vee q)}$ is
\[
(-1)^{r_1 -1} C_{r_1 - 1} \cdots (-1)^{r_k-1} C_{r_k-1}
\]
where the $i^{th}$ cycle $c_i$ has $r_i$ elements. 
\end{remark}

The reason for introducing $\Wg$ is its use in computing
matrix expectations. For pairings $p, q \in \cP_2(n)$, $p
\delta q \delta$ is a pairing of $[\pm n]$. For a pairing
$r$ of $[\pm n]$ and $i_1, i_{-1},\dots, i_n, i_{-n} \in
       [d]$ we let $\delta_r^i = 1$ if $i_s = i_t$ whenever
       $(s, t)$ is a pair of $r$ and 0 otherwise.

\begin{theorem}[{\cite[Cor. 3.4]{cs}}]
\label{theorem:weingartenexpansion}
When $n$ is even
\[
\E(o_{i_1i_{-1}} \cdots o_{i_ni_{-n}}) = \sum_{p,q \in
  \cP_2(n)} \langle \Wg(p), q \rangle\, \delta_{p\delta q
  \delta}^i.
\]
When $n$ is odd, $\E(o_{i_1i_{-1}} \cdots o_{i_ni_{-n}}) = 0$.
\end{theorem}

\begin{corollary}\label{corollary:haarorthogonaldistribution}
Let $O$ be a $d \times d$ Haar distributed orthogonal matrix
and $m$ a non-zero integer. Then
\[
\lim_{d \rightarrow \infty}
\E(\tr(O^m)) = 0.
\]
\end{corollary}

\begin{proof}
Let $\gamma \in S_m$ be the permutation with the one cycle
$(1, 2, 3, \dots,\ab m)$. If $m$ is odd then $\E(\Tr(O^m)) =
0$. So suppose that $m$ is even. First let us consider
\begin{eqnarray*}
\E(\Tr(O^m)) & = & \sum_{i_1, \dots, i_m=1 }^d
\E(o_{i_1i_{\gamma(1)}} \cdots o_{i_ni_{\gamma(n)}}) \\ & =
& \sum_{i_1, \dots, i_m} \sum_{p,q \in \cP_2(m)} \langle
\Wg(p), q \rangle \delta_p^i \delta_q^{i\gamma},
\end{eqnarray*}
where $i\gamma$ is the $m$-tuple $(i_{\gamma(1)}, \dots,
i_{\gamma(m)})$. Now $\delta_q^{i\gamma} = \delta_{\gamma q
  \gamma^{-1}}^i$. Thus $\delta_p^i \delta_q^{i\gamma} = 1$
only when $i$ is constant on the blocks of $p \vee \gamma
q\gamma^{-1}$. Hence
\[
\E(\Tr(O^m)) = \sum_{p, q \in \cP_2(m)} \langle \Wg(p),
q\rangle d^{\#(p \vee \gamma q\gamma^{-1})}.
\] 
Thus $\E(\Tr(O^m)) = \rO(d^{-m + \#(p \vee q) + \#(p \vee
  \gamma q\gamma^{-1})})$. But $-m + \#(p \vee q) + \#(p
\vee \gamma q\gamma^{-1})\leq 0$. Hence $\lim_{d \rightarrow
  \infty} \E(\tr(O^m)) = 0$.
\end{proof}

\begin{notation}\label{notation:krewerascomplement}
Let $\gamma \in S_n$ be a permutation of $[n]$ but, as in
Notation \ref{notation:permutationextension}, considered as
a permutation of $[\pm n]$ by setting $\gamma(-k) = -k$ for
$1 \leq k \leq n$. Given $\epsilon \in \bZ_2^n$ and $p, q
\in \cP_2(n)$ we consider the pairing of $[\pm n]$ given by
$p \mskip 2mu\cdot_\epsilon q = (\gamma\delta)^{-1}
\delta_\epsilon p \delta q \delta \delta_\epsilon (\gamma
\delta)$ of $[\pm n]$. By Lemma
\ref{lemma:pairingtopermutation} there is a permutation
$\pi_{p \mskip 2mu\cdot_\epsilon q} \in S_n$ and
$\eta_{p \mskip 2mu\cdot_\epsilon q} \in \bZ_2^n$ such that
\[
\mathop{\sum_{i_1, i_{-1}, \dots, i_{n}, i_{-n}=1}}_{i = i
  \circ p \mskip 2mu\cdot_\epsilon q}^d a^{(1)}_{i_1i_{-1}}
a^{(2)}_{i_2i_{-2}} \cdots a^{(k)}_{i_{k}i_{-k}} \cdots
a^{(n)}_{i_{n}i_{-n}} = \Tr_{(\pi_{p \mskip
    2mu\cdot_\epsilon q}, \eta_{p \mskip 2mu\cdot_\epsilon
    q})} (A_1, \dots, A_n).
\]
Note that $p$ is a pairing of $[n]$, $\delta q \delta$ is a
pairing of $[-n]$ and so $p\delta q \delta$ is a pairing of
$[\pm n]$. If we adopt the notation $\gamma_- =
\delta\gamma\delta$ then $(p \mskip 2mu \cdot_\epsilon q)
\delta = \gamma_-^{-1} \delta_\epsilon q \delta p
\delta_\epsilon \gamma$. Recall from the proof of Lemma
\ref{lemma:pairingtopermutation} that $\pi_{p \mskip
  2mu\cdot_\epsilon q}$ was obtained by writing $(p \mskip
2mu\cdot_\epsilon q) \delta$ as a product of cycles and
taking one cycle of each pair $\{c, c'\}$.  After this
choice has been made $\eta_{p \mskip 2mu\cdot_\epsilon q}$
records the position of the minus signs.

\end{notation}

\begin{proposition}\label{proposition:exact_formula}
Let $O$ be a Haar distributed $d \times d$ random orthogonal
matrix and $\{Y_1, \dots, Y_n\}$ $d \times d$ random
matrices which are independent from $O$ and whose entries
have moments of all orders. Let $\gamma \in S_n$, $\epsilon
\in \bZ_2^n$ and suppose $d \geq n$.

\begin{eqnarray*}
\lefteqn{
\E(\Tr_\gamma(O^{\epsilon_1}Y_1, \dots,
O^{\epsilon_n}Y_n))}\\ & = & \sum_{p,q \in \cP_2(n)} \langle
  \Wg(p), q\rangle \E(\Tr_{(\pi_{p \mskip 2mu\cdot_\epsilon
      q}, \eta_{p \mskip 2mu\cdot_\epsilon q})}(Y_1, \cdots,
  Y_n))
\end{eqnarray*}
\end{proposition}

\begin{proof}
\begin{eqnarray*}\lefteqn{
\E(\Tr_\gamma(O^{\epsilon_1} Y_1 \cdots O^{\epsilon_n}
Y_n))} \\ & = & \sum_{j_1, \dots, j_{-n}=1}^d
  \E(o^{(\epsilon_1)}_{j_1j_{-1}} \cdots
  o^{(\epsilon_n)}_{j_{n}j_{-n}})\ \E(y^{(1)}_{j_{-1}j_{\gamma(1)}}
  \cdots y^{(n)}_{j_{-n}j_{\gamma(n)}})
\end{eqnarray*}
Now for notational convenience let $\epsilon(k) =
\epsilon_{|k|}k$ and let $l_k = j_{\epsilon(k)}$, then
$o^{(\epsilon_k)}_{j_kj_{-k}} = o_{l_k l_{-k}}$. Thus
\begin{equation}\label{equation:simpleorthogonalexpansion}
\E(o^{(\epsilon_1)}_{j_1j_{-1}} \cdots
o^{(\epsilon_n)}_{j_nj_{-n}}) = \E(o_{l_1l_{-1}}\cdots
o_{l_n l_{-n}}) = \sum_{p, q \in \cP_2(n)} \langle \Wg(p), q
\rangle \delta^l_{p \delta q\delta},
\end{equation}
where $\delta^l_{p\delta q\delta} = 1$ if $l = l \circ
p\delta q\delta$. Also $y^{(k)}_{j_{-k}j_{k+1}} =
y^{(k)}_{l_{-\epsilon(k)} l_{\epsilon\gamma(k)}}$. Hence we
have
\begin{eqnarray*}\lefteqn{
\E(\Tr_\gamma(O^{\epsilon_1} Y_1 \cdots O^{\epsilon_n}
Y_n))} \\ & = & \sum_{l_1, \dots, l_{-n}} \E(o_{l_1l_{-1}}
  \cdots
  o_{l_nl_{-n}})\ \E(y^{(1)}_{l_{-\epsilon(1)}l_{\epsilon\gamma(1)}}
  \cdots y^{(n)}_{l_{-\epsilon(n)}l_{\epsilon\gamma(n)}}
  )\\ & = & \sum_{p,q \in \cP_2(n)} \langle \Wg(p), q
  \rangle \mathop{\sum_{l_1, \dots, l_{-n}}}_{l = l\circ p
    \delta q \delta}
  \E(y^{(1)}_{l_{-\epsilon(1)}l_{\epsilon\gamma(1)}} \cdots
  y^{(n)}_{l_{-\epsilon(n)}l_{\epsilon\gamma(n)}}).
\end{eqnarray*}

Let $i = l \circ \epsilon \gamma\delta$. Then $i_1 =
l_{-\epsilon(1)}, i_{-1} = l_{\epsilon\gamma(1)}, \dots ,
i_n = l_{-\epsilon(n)}, i_{-n} =
l_{\epsilon\gamma(n)}$. Thus as $p \mskip 2mu\cdot_\epsilon
q = \delta \gamma^{-1} \delta_\epsilon p \delta q \delta
\delta_\epsilon \gamma \delta$ we have
\[
\mathop{\sum_{l_1, \dots, l_{-n}}}_{l = l\circ p \delta q
  \delta} \E(y^{(1)}_{l_{-\epsilon(1)}l_{\epsilon\gamma(1)}}
\cdots y^{(n)}_{l_{-\epsilon(n)}l_{\epsilon\gamma(n)}}) =
\mathop{\sum_{i_1, \dots, i_{-n}}}_{i = i \circ p \mskip
  2mu\cdot_\epsilon q} \E(y^{(1)}_{i_1 i_{-1}} \cdots
y^{(n)}_{i_ni_{-n}})
\]

So
\begin{eqnarray*}\lefteqn{
\E(\Tr_\gamma(O^{\epsilon_1} Y_1 \cdots O^{\epsilon_n}
Y_n))} \\ & = & \sum_{p,q \in \cP_2(n)} \langle \Wg(p), q
  \rangle \mathop{\sum_{i_1, \dots, i_{-n}}}_{i = i \circ
    p \mskip 2mu\cdot_\epsilon q} \E(y^{(1)}_{i_1 i_{-1}}
  \cdots y^{(n)}_{i_ni_{-n}})\\ & = & \sum_{p,q \in
    \cP_2(n)} \langle \Wg(p), q \rangle \E(
  \Tr_{(\pi_{p \mskip 2mu\cdot_\epsilon q}, \eta_{p \mskip
      2mu\cdot\epsilon q})}(Y_1, \dots, Y_n)).
\end{eqnarray*}
\end{proof}

We shall need a special case of this result in section
\ref{section:mainresults}. Let us say that a permutation
$\pi$ is \textit{parity preserving} if for all $k$, $\pi(k)$
and $k$ have the same parity.

\begin{lemma}\label{lemma:evenodd}
Let $n_1, n_2, \dots, n_r$ be even positive integers and $n
= n_1 + \cdots + n_r$. Let $\gamma = (1,2, \dots,
n_1)(n_1+1, \dots, n_1 + n_2) \cdots (n_1 + \cdots +
n_{r-1}+1, \dots, n_1 + \cdots + n_r) \in S_n$. Suppose that
$\epsilon \in \bZ_2^n$ is such that $\epsilon_k =
(-1)^{k+1}$. Then for all $p, q \in \cP_2(n)$, $\pi_\peq$ is
parity preserving.
\end{lemma}

\begin{proof}
We first show that $\peq = \delta \gamma^{-1}
\delta_\epsilon p \delta q \delta \delta_\epsilon \gamma
\delta$ is parity preserving. By direct computation we have
the following.
\begin{align*}
\peq(2k-1)& = \left\{ 
\begin{array}{cc}
-\gamma^{-1}(q(2k-1)) & q(2k-1) \mbox{\ is even},\\
q(2k-1) & q(2k-1) \mbox{\ is odd};
\end{array}\right. \\
\peq(-(2k-1)) &= \left\{ 
\begin{array}{cc}
-\gamma^{-1}(q(2k)) & q(2k) \mbox{\ is even,}\\
q(2k) & q(2k) \mbox{\ is odd.}
\end{array}\right.
\end{align*}
Note that since $\gamma$ always reverses the parity of its
argument, all four possible outcomes are odd. Thus $\peq$
takes odd numbers to odd numbers. Since $\peq$ is a
permutation it must then take even numbers to even
numbers. Indeed
\begin{align*}
\peq(2k)& = \left\{ 
\begin{array}{cc}
p(2k) & p(2k) \mbox{\ is even}\\
-\gamma^{-1}(p(2k)) & p(2k) \mbox{\ is odd}
\end{array}\right. \\
\peq(-(2k)) &= \left\{ 
\begin{array}{cc}
p(\gamma(2k)) & p(\gamma(2k)) \mbox{\ is even}\\
-\gamma^{-1}(p(\gamma(2k))) & p(\gamma(2k)) \mbox{\ is odd}
\end{array}\right.
\end{align*}
Now $\delta(k) = -k$ is parity preserving, thus so is
$(\peq)\delta$. Finally $\pi_\peq$ is obtained by choosing
one representative of each pair $\{c, \delta c^{-1}
\delta\}$ of $(\peq)\delta$, and taking the absolute value
of each entry. This means that each cycle will consist of
integers of the same parity. Hence $\pi_\peq$ is parity
preserving.
\end{proof}

We wish to extend the conclusion of Proposition
\ref{proposition:exact_formula} to case where some of the
$Y$'s are not interleaved by orthogonal matrices.

\begin{proposition}
\label{proposition:exact_formulafirstextension}
Let $O$ be a Haar distributed $d \times d$ random orthogonal
matrix and $\{Y_1, \dots, Y_n\}$ $d \times d$ random
matrices which are independent from $O$ and whose entries
have moments of all orders. Let $1 \leq m \leq n$, $\gamma
\in S_m$, $\epsilon \in \bZ_2^m$ and suppose $d \geq m$.

\begin{eqnarray*}
\lefteqn{
\E(\Tr_\gamma(O^{\epsilon_1}Y_1, \dots,O^{\epsilon_m}Y_m)
\Tr(Y_{m+1}) \cdots \Tr(Y_n))}\\ 
& = & 
\sum_{p,q \in \cP_2(m)} \kern-1em \langle
  \Wg(p), q\rangle \E(\Tr_{(\pi_{p \mskip 2mu\cdot_\epsilon
      q}, \eta_{p \mskip 2mu\cdot_\epsilon q})}(Y_1, \cdots,
  Y_m) \Tr(Y_{m+1}) \cdots \Tr(Y_n))
\end{eqnarray*}

\end{proposition}

\begin{proof}
The proof is the same as for Proposition
\ref{proposition:exact_formula} except that we append the
random variable $\Tr(Y_{m+1}) \cdots \Tr(Y_n)$ to the right
hand side of each expression.
\end{proof}

We now wish to extend the conclusion of Proposition
\ref{proposition:exact_formula} in another way, namely to
the case of independent Haar distributed orthogonal
matrices. Suppose $\{O_1, \dots, O_s\}$ are independent Haar
distributed $d \times d$ orthogonal matrices, with the
$(i,j)$ entry of $O_k$ denoted $o{}_{(k)ij}$. We shall need
a expression for $\E(o_{(k_1)i_1i_{-1}} o_{(k_2)i_2i_{-2}}
\cdots o_{(k_n)i_ni_{-n}})$ extending that given in Theorem
\ref{theorem:weingartenexpansion}.

\begin{notation}
Given an $n$-tuple $(i_1, i_2, \dots, i_n)$ of integers in
$[s]$ we let $\ker(i)$ be the partition of $[n]$ such that
$i_r = i_s$ where $r$ and $s$ are in the same block of
$\ker(i)$ and $i_r \not= i_s$ when $r$ and $s$ are in
different blocks of $\ker(i)$.

Let $\cU \in \cP(n)$ be a partition of $[n]$ and $p \in
\cP_2(n)$ be a pairing such that each pair of $p$ lies in
some block of $\cU$. We shall denote this by $p \leq
\cU$. If we write the blocks of $\cU$ as $\{U_1, \dots,
U_r\}$, then the pairs of $p$ that lie in $U_i$ form a
pairing of $U_i$ which we shall denote by $p_{|U_i}$ or just
$p_i$ when convenient.

If we have a partition $\cU$ and pairings $p, q \in
\cP_2(n)$ with $p,q \leq \cU$ then we let
\[
\Wg(\cU,p,q) =
\langle \Wg(p_1), q_1 \rangle \cdots
\langle \Wg(p_r), q_r \rangle.
\]
\end{notation}

\begin{remark}\label{remark:subleadingorder}
Note that since $\Wg$ is not multiplicative, $\Wg(\cU, p,q)$
and $\langle \Wg(p), q\rangle$ are different. However by
Remark \ref{remark:leadingorder} we see that when $p, q \leq
\cU$ then $\Wg(\cU, p, q) - \langle \Wg(p), q \rangle =
\rO(d^{-n + \#(p \vee q)-1})$ as the leading terms in both
expressions are the same.
\end{remark}

\begin{lemma}
Suppose $\{O_1, \dots, O_s\}$ are independent Haar
distributed $d \times d$ orthogonal matrices. Let the
$(i,j)$ entry of $O_k$ be denoted $o_{(k)i,j}$. Given an
$n$-tuple $(k_1, \dots, k_n)$ in $[s]$ then
\[
\E(o_{(k_1)i_1i_{-1}} o_{(k_2)i_2i_{-2}} \cdots o_{(k_n)i_ni_{-n}})
=
\mathop{\sum_{p,q \in \cP_2(n)}}_{p, q \leq \ker(k)}
\Wg(\ker(k), p,q)\, \delta_{p\delta q \delta}^i.
\]

\end{lemma}

\begin{proof}
We can write $\E(o_{(k_1)i_1i_{-1}} o_{(k_2)i_2i_{-2}}
\cdots o_{(k_n)i_ni_{-n}})$ as a product of expectations,
one for each block of $\ker(k)$. For each block $U_j$ of
$\ker(k)$ we get a factor $ \sum_{p_j, q_j \in \cP_2(U_i)}
\langle \Wg(p_j), q_j \rangle d_{p_j \delta q_j
  \delta}^{i_j} $ where $i_j$ is the restriction of $i$ to
the block $U_j$. Taking the product of these terms we get
$\ds\mathop{\sum_{p,q \in \cP_2(n)}}_{p, q \leq \ker(k)}
\Wg(\ker(k), p,q)\, \delta_{p\delta q \delta}^i$.
\end{proof}

\begin{proposition}\label{proposition:exactformulaextended}
Let $\{O_1, \dots, O_s\}$ be independent Haar distributed $d
\times d$ orthogonal matrices and $\{Y_1, \dots, Y_n\}$
$d \times d$ random matrices which are independent from $\{O_1, \dots,
O_s\}$ and whose entries have moments of all orders. Let
$\gamma \in S_n$, $\epsilon \in \bZ_2^n$ and suppose $d \geq
n$. For each $n$-tuple $(k_1, \dots, k_n)$ in $[s]$ we have

\begin{eqnarray*}
\lefteqn{
\E(\Tr_\gamma(O_{k_1}^{\epsilon_1}Y_1, \dots,
O_{k_n}^{\epsilon_n}Y_n))}\\ & = & 
\mathop{\sum_{p,q \in \cP_2(n)}}_{p,q \leq \ker(k)}
  \Wg(\ker(k), p, q) \E(\Tr_{(\pi_{p \mskip 2mu\cdot_\epsilon
      q}, \eta_{p \mskip 2mu\cdot_\epsilon q})}(Y_1, \cdots,
  Y_n)).
\end{eqnarray*}

\end{proposition}

\begin{proof}
The only point where the proof differs from the proof of
Proposition \ref{proposition:exact_formula} is in Equation
\ref{equation:simpleorthogonalexpansion}, which we replace
by
\begin{eqnarray*} \lefteqn{
\E(o^{(\epsilon_1)}_{(k_1)j_1j_{-1}} \cdots
o^{(\epsilon_n)}_{(k_n)j_nj_{-n}}) } \\
& = &
\E(o_{(k_1)l_1l_{-1}}\cdots
o_{(k_n)l_n l_{-n}}) = 
\mathop{\sum_{p, q \in \cP_2(n)}}_{p,q \leq \ker(k)}
\Wg(\ker(k), p, q) \delta^l_{p \delta q\delta}.
\end{eqnarray*}
The remainder of the proof is unchanged. 
\end{proof}

\section{A Lemma on Spoke Diagrams}\label{section:spokediagrams}

At several points later on we shall wish to know that a
given permutation represents a spoke diagram (see Figure
\ref{figure:bothspokediagrams}). Lemma
\ref{lemma:standardspokediagram} identifies standard spoke
diagrams and Lemma \ref{lemma:reversedspokediagram}
identifies reversed spoke diagrams.

\begin{lemma}\label{lemma:preliminaryspokediagram}
Suppose $\gamma \in S_n$ is a permutation, $p \in \cP_2(n)$
a pairing, and $\epsilon \in \bZ_2^n$ an assignment of
signs, are such that $\pi_\pep$ is a pairing. Let $(r, s)
\in p$ be a pair of $p$.

\begin{enumerate}

\item 
If $\epsilon_r = -\epsilon_s$ then $(\gamma^{-1}(r),
\gamma(s)) \in p$, $(\gamma^{-1}(r), s) \in \pi_\pep$, and
$\epsilon_{\gamma^{-1}(r)} = -\epsilon_{\gamma(s)}$.

\item 
If $\epsilon_r = \epsilon_s$ then $(\gamma^{-1}(r),
\gamma^{-1}(s)) \in p$, $(\gamma^{-1}(r), \gamma^{-1}(s))
\in \pi_\pep$, and $\epsilon_{\gamma(r)} =
\epsilon_{\gamma(s)}$.

\end{enumerate}
\end{lemma}

\begin{proof} (\textit{i})
Let us suppose that $\epsilon_r = -\epsilon_s$. Since $(r,
s) \in p$ and $\epsilon_r = -\epsilon_s$ we have
\[
(r,-s), (-r,s) \in \delta_\epsilon p \delta p \delta
\delta_\epsilon.
\] 
Since $\pi_\pep$ is a pairing, $(\pep)\delta$ is also a
pairing --- recall that $\pep = (\gamma\delta)^{-1}
\delta_\epsilon p \delta p \delta \delta_\epsilon(\gamma
\delta)$. Also
\[
(p \mskip 2mu\cdot_\epsilon p)\delta(\gamma^{-1}(r))
=(\gamma\delta)^{-1} (\delta_\epsilon p \delta p \delta
\delta_\epsilon) (\gamma\delta) \delta (\gamma^{-1}(r)) = s.
\] 
Thus $(\gamma^{-1}(r), s) \in (p \mskip 2mu\cdot_\epsilon
p)\delta$, because $(\pep)\delta$ is a pairing. Since both
$\gamma^{-1}(r), s \in [n]$ we have that $(\gamma^{-1}(r),
s) \in \pi_\pep$. Moreover $(\gamma^{-1}(r),\ab s) \in
\pi_\pep$ and so $(p \mskip 2mu\cdot_\epsilon p)\delta(s) =
\gamma^{-1}(r)$. Unwinding this equation we have
\[
p\delta p \delta (\epsilon_{\gamma(s)} \gamma(s)) = -
\epsilon_{\gamma^{-1}(r)} \gamma^{-1}(r).
\] 
Since $p\delta p \delta$, as a permutation, doesn't change
the sign of its argument, we have $\epsilon_{\gamma(s)} =
-\epsilon_{\gamma^{-1}(r)}$. Thus $p\delta p \delta
(\gamma(s)) = \gamma^{-1}(r)$, and we are left with
$(\gamma^{-1}(r), \gamma(s))$ is a cycle of $p$,
$(\gamma^{-1}(r), s) \in \pi_\pep$, and
$\epsilon_{\gamma(s)} = -\epsilon_{\gamma^{-1}(r)}$ as
required.

(\textit{ii}) Let us suppose that $\epsilon_r =
\epsilon_s$. Since $(r, s) \in p$ and $\epsilon_r =
\epsilon_s$ we have
\[
(r,s), (-r,-s) \in \delta_\epsilon p \delta p \delta
\delta_\epsilon.
\] 
Since $\pi_\pep$ is also a pairing, $(\pep)\delta$ is a
pairing. Also
\[
(p \mskip 2mu\cdot_\epsilon p)\delta(\gamma^{-1}(r))
=(\gamma\delta)^{-1} (\delta_\epsilon p \delta p \delta
\delta_\epsilon) (\gamma\delta) \delta (\gamma^{-1}(r)) =
-\gamma^{-1}(s).
\] 
Thus $(\gamma^{-1}(r), -\gamma^{-1}(s))$ is a pair of
$(p \mskip 2mu\cdot_\epsilon p)\delta$. Thus
$(\gamma^{-1}(r), \gamma^{-1}(s)) \in \pi_\pep$. Moreover
$(p \mskip 2mu\cdot_\epsilon p)\delta(-\gamma^{-1}(s)) =
\gamma^{-1}(r)$. Unwinding the equation $(p \mskip
2mu\cdot_\epsilon p)\delta(-\gamma^{-1}(s)) =
\gamma^{-1}(r)$ we have
\[
p\delta p \delta (-\epsilon_{\gamma(s)} \gamma(s)) =
-\epsilon_{\gamma^{-1}(r)} \gamma^{-1}(r).
\] 
Since $p\delta p \delta$, as a permutation, doesn't change
the sign of its argument, we have $\epsilon_{\gamma^{-1}(r)}
= \epsilon_{\gamma^{-1}(s)}$. Thus $p\delta p \delta
(\gamma^{-1}(s)) = \gamma^{-1}(r)$, and we are left with
$(\gamma^{-1}(r), \gamma^{-1}(s))$ is a cycle of $p$,
$\epsilon_{\gamma^{-1}(r)} = \epsilon_{\gamma^{-1}(s)}$, and
$(\gamma^{-1}(r), \gamma^{-1}(s)) \in \pi_\pep$ as claimed.

\end{proof}

\begin{lemma}\label{lemma:standardspokediagram}
Let $\gamma$ be the permutation with the two cycles $(1,
\dots, m)\ab(m+1, \dots, m+n)$, let $\epsilon \in
\bZ_2^{m+n}$, and let $p\in \cP_2(m+n)$ be a pairing such
that
\begin{enumerate}
\item
$p \vee \gamma = 1_{m+n}$, i.e. at least one of cycle of $p$
  connects the two cycles of $\gamma$;

\item for some $(r, s) \in p$ we have $\epsilon_r = -
  \epsilon_s$; and

\item 
$\pi_\pep$ is a pairing.
\end{enumerate}
Then $m = n$, $p$ and $\pi_\pep$ are standard spoke
diagrams, and there is $l$, which we make take to be
$\gamma^{-r}(s)$ if we assume that $1 \leq r \leq m$, such
that

\begin{enumerate}

\item[\textit{a})] every cycle of $p$ is of the form $(k,
  \gamma^{-k}(l))$ for $1 \leq k \leq m$, and

\item[\textit{b})] every cycle of $\pi_\pep$ is of the form
  $(k, \gamma^{-k-1}(l))$ for $1 \leq k \leq m$,
 
\item[\textit{c})] 
 and $\epsilon_r = -\epsilon_s$ for all
$(r, s) \in p$,

\item[\textit{d})]  $\eta_k =1$ for all $k$.
\end{enumerate}
\end{lemma}

\begin{proof}
Let $(r, s) \in p$, i.e. $(r, s)$ is a cycle of $p$, and
suppose $\epsilon_r = -\epsilon_s$. By using induction on
Lemma \ref{lemma:preliminaryspokediagram} we know that for
all $k$, $(\gamma^{-k}(r), \gamma^k(s)) \in p$,
$\epsilon_{\gamma^{-k}(r)} = -\epsilon_{\gamma^k(s)}$, and
$(\gamma^{-k}(r), \gamma^{k-1}(s)) \in \pi_\pep$. Recall
that in the proof of Lemma
\ref{lemma:preliminaryspokediagram} (\textit{i}) we showed
that $(\gamma^{-1}(r), s) \in (\pep)\delta$. This implied
that $(\gamma^{-1}(r), s) \in \pi_\pep$ and that
$\eta_{\gamma^{-1}(r)} = \eta_s = 1$. By our induction
argument we have that $\eta_k = 1$ for all $k$.

By assumption, $p$ has at least one pair $(r, s)$ that
connects the cycles of $\gamma$; and so by what we have just
observed, all cycles of $p$ connect the two cycles of
$\gamma$. This implies $m = n$, and all cycles of $p$ are of
the form $(k, \gamma^{-k}(l))$, where $l = \gamma^{r-1}(s)$,
assuming $\gamma^{-r}(r) = m$. Moreover, both $p$ and
$\pi_\pep$ are spoke diagrams, i.e. non-crossing annular
pairings of an $(m,m)$-annulus with all pairs connecting the
two circles; see Figure~\ref{figure:bothspokediagrams}.
\end{proof}

\begin{lemma}\label{lemma:reversedspokediagram}
Let $\gamma$ be the permutation with the two cycles $(1,
\dots, m)\ab(m+1, \dots, m+n)$, let $\epsilon \in
\bZ_2^{m+n}$, and let $p\in \cP_2(m+n)$ be a pairing such
that
\begin{enumerate}
\item
$p \vee \gamma = 1_{m+n}$, i.e. at least one of cycle of $p$
  connects the two cycles of $\gamma$;

\item for some $(r, s) \in p$ we have $\epsilon_r =  \epsilon_s$; and

\item 
$\pi_\pep$ is a pairing.
\end{enumerate}
Then $m = n$, $p$ and $\pi_\pep$ are reversed spoke
diagrams, and there is $l$, which we make take to be
$\gamma^{-r}(s)$ if we assume that $1 \leq r \leq m$, such
that

\begin{enumerate}

\item[\textit{a})] every cycle of $p$ is of the form $(k,
  \gamma^{k}(l))$ for $1 \leq k \leq m$, and

\item[\textit{b})] every cycle of $\pi_\pep$ is of the form
  $(k, \gamma^k(l))$ for $1 \leq k \leq m$,
 
\item[\textit{c})] 
 and $\epsilon_r = \epsilon_s$ for all
$(r, s) \in p$,

\item[\textit{d})] $\eta_k = -1$ for all $k \in [m]$.
\end{enumerate}
\end{lemma}

\begin{proof}
Let $(r, s) \in p$, i.e. $(r, s)$ is a cycle of $p$, and
suppose $\epsilon_r = \epsilon_s$. By using induction on
Lemma \ref{lemma:preliminaryspokediagram} we know that for
all $k$, $(\gamma^{-k}(r), \gamma^{-k}(s)) \in p$,
$\epsilon_{\gamma^{-k}(r)} = \epsilon_{\gamma^{-k}(s)}$, and
$(\gamma^{-k}(r), \gamma^{-k}(s)) \in \pi_\pep$. Recall that
in the proof of Lemma \ref{lemma:preliminaryspokediagram}
(\textit{ii}) we showed that $(\gamma^{-1}(r),
-\gamma^{-1}(s)) \in (\pep)\delta$. This implied that
$(\gamma^{-1}(r), \gamma^{-1}(s)) \in \pi_\pep$ and that
$\eta_{\gamma^{-1}(r)} = -1$. By our induction argument we
have that $\eta_k = -1$ for all $k\in [m]$.

By assumption, $p$ has at least one pair $(r, s)$ that
connects the cycles of $\gamma$; and so by what we have just
observed, all cycles of $p$ connect the two cycles of
$\gamma$. This implies $m = n$, and all cycles of $p$ are of
the form $(k, \gamma^{k}(l))$, where $l = \gamma^{-r}(s)$,
assuming $\gamma^{-r}(r) = m$. Moreover, both $p$ and
$\pi_\pep$ are spoke diagrams, i.e. non-crossing annular
pairings of an $(m,m)$-annulus with all pairs connecting the
two circles; see Figure~\ref{figure:bothspokediagrams}.
\end{proof}

\begin{corollary}\label{corollary:maximalconectivity}
Let $\gamma \in S_n$ be a permutation, $p \in \cP_2(n)$ a
pairing, and $\epsilon \in \bZ_2^n$ an assignment of
signs. Suppose that $\pi_\pep$ is a pairing then each block
of $p \vee \gamma$ contains at most two cycles of $\gamma$.
\end{corollary}

\begin{proof}
We saw in Lemma \ref{lemma:preliminaryspokediagram} that
when $p$ connects a pair of cycles of $\gamma$ these two
cycles form a spoke diagram. So a block of $p \vee \gamma$
can contain at most two cycles of $\gamma$. 
\end{proof}


\section{Real Second Order Freeness}
\label{section:realsecondorderfreeness}

Let us recall the definition of real second order freeness
from Redelmeier \cite[\S 1]{r2}. We begin with the concept of a
real second order non-commutative probability space.

\begin{definition}
Let $\cA$ be an algebra over $\bC$ and with an
anti-auto\-mor\-phism of order 2 denoted by $a \mapsto
a^t$. Suppose that $\phi : \cA \rightarrow \bC$ is a tracial
state and $\phi_2 : \cA \times \cA \rightarrow \bC$ is a
bi-trace, i.e. $\phi_2$ is bilinear and tracial in each
entry. Moreover we assume that $\phi_2(1, a) = \phi(a, 1) =
0$, $\phi(a^t) = \phi(a)$ and $\phi_2(a^t, b) = \phi_2(a,
b^t) = \phi_2(a, b)$ for all $a, b \in \cA$.  Then $(\cA,
\phi, \phi_2, t)$ is a \textit{real second order
  non-commutative probability space}.

\end{definition}

\begin{notation}

Let unital subalgebras $\cA_1,\dots,\cA_r\subset\cA$ be given.

\begin{enumerate}

\item
We say that a tuple $(a_1,\dots,a_n)$ of elements from $\cA$
is \emph{cyclically alternating} if, for each $i$, there is
$j_i \in \{1, \dots, r\}$ such that $a_i \in \cA_{j_i}$ and,
if $n\geq 2$, we have $j_k \not= j_{k+1}$ for all $k=1,
\dots, n$. We count indices in a cyclic way modulo $n$,
i.e., for $k=n$ the equation above means $j_n \not= j_1$.

\item
We say that a tuple $(a_1,\dots,a_n)$ of elements from $\cA$
is \emph{centred} if we have
$$\ff(a_i)=0\qquad\text{for all $i=1,\dots,n$.}$$
\end{enumerate}

\end{notation}

\begin{definition}\label{definition:secondorderfreeness}
Let $(\cA, \phi, \phi_2, t)$ be a real second order
non-commuta\-tive probability space and suppose that we have
unital subalgebras $\cA_1, \dots, \ab\cA_n$ that are
invariant under $a \mapsto a^t$. We say that $\cA_1, \dots,
\cA_n$ are real free of second order if (see figure \ref{figure:twospokes})
\begin{enumerate}

\item
the subalgebras $\cA_1, \dots, \cA_n$ are free with respect
to $\phi$;

\item for every $a_1, \dots, a_m \in \cA$ and $b_1, \dots,
  b_n \in \cA$ such that $(a_1, \dots,\ab a_m) $ and $(b_1,
  \dots, b_n)$ are centred and cyclically alternating, we
  have
\begin{enumerate}
\item[{\it a})] $\phi_2(a_1 \cdots a_m, b_1 \cdots b_n) =
  0$, if $m \not = n$ or if $m = n = 1$ and $a_1$ and $b_1$ are from different subalgebras;

\item[{\it b})] for $m = n> 1$ we have, taking all indices
  modulo $n$
\begin{equation}\label{equation:secondorderdefinition}
\phi_2(a_1 \cdots a_n, b_1 \cdots b_n) = \sum_{k=1}^n
\prod_{i=1}^n \bigg(\phi(a_i b_{k-i}) + \phi(a_i b_{i-k}^t)
\bigg).
\end{equation}

\end{enumerate}
\end{enumerate}
\end{definition}

\begin{figure}
\hfill\includegraphics{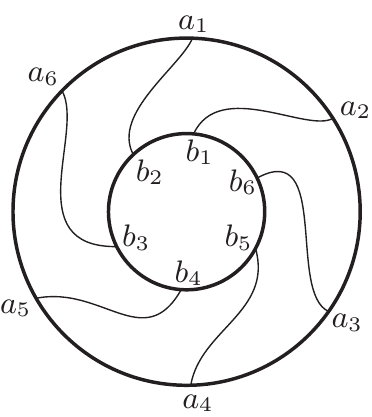}\hfill\includegraphics{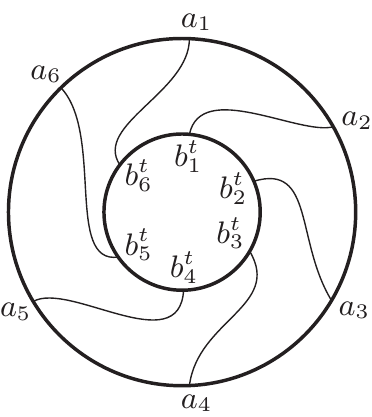}\hfill\hbox{}
\caption{\label{figure:twospokes} The terms on the right
  hand side of equation
  (\ref{equation:secondorderdefinition}) are sums over all
  spoke diagrams. In the diagram on the left the circles
  have the opposite orientation; we put the $a$'s on on
  circle and the $b$'s on the other. This gives the first
  term on the right hand side of
  (\ref{equation:secondorderdefinition}). In the circle on
  the right the two circles have the same orientation and we
  put `$b^t$'s on the inside circle. This gives the second
  term on the right hand side of
  (\ref{equation:secondorderdefinition}).  }
\end{figure}

\begin{notation}
Let $p \in \bC[x_1, \dots, x_s, x_1^t, \dots, x_s^t]$ be a
polynomial in the non-commuting variables $\{x_1, \dots,
x_s, x_1^t, \dots, x_s^t\}$ and $A_1, \dots , A_s$ be $d
\times d$ matrices. By $p(A_1, \dots, A_s)$ we mean the
matrix obtained by replacing $x_i$ by $A_i$ and $x_i^t$ by
$A_i^t$ in $p$. Similarly if $(\cA, \phi, \phi_2, t)$ is a
real second order non-commutative probability space then by
$p(a_1, \dots, a_s)$ we mean the random variable in $\cA$
obtained by replacing $x_i$ by $a_i$ and $x_i^t$ by $a_i^t$.

\end{notation}

\begin{remark}
Expanding on the notation in equation
(\ref{equation:permutation}) we define, for a permutation
$\pi \in S_n$ and $a_1, \dots, a_n \in \cA$, $\phi_\pi(A_1,
\dots, a_n)$ as below.
\[
\phi_\pi(a_1, \dots, a_n) = \mathop{\prod_{c \in \pi}}_{c =
  (i_1, \dots, i_k)} \phi(a_{i_1} \cdots a_{i_k}),\] where
the product is over all cycles $c$ of $\pi$ and for each
cycle $c = (i_1, \dots, i_k)$ we get the factor
$\phi(a_{i_1}, \dots, a_{i_k})$. This makes $\phi_\pi$ a
$n$-linear functional.

With this notation we can write equation
(\ref{equation:secondorderdefinition}) in a simpler way:
\begin{multline*}
\phi_2(a_1 \cdots a_n, b_1 \cdots b_n) \\ = \sum_{\pi \in
  \spoke{n}} \phi_\pi(a_1, \dots, a_n,
b_1, \dots, b_n) \\ \mbox{} + \sum_{\pi \in
  \rspoke{n}} \phi_\pi(a_1, \dots, a_n, b_1^t,
\dots, b_n^t),
\end{multline*}
where, recall, $\spoke{n}$ denotes the set of standard spoke
diagrams and $\rspoke{n}$ denotes the set of standard spoke
diagrams.
\end{remark}

We shall need to use the associativity of real second order
freeness. Let us recall how this works in the first order
case \cite{vdn}. Suppose that we have unital subalgebras
$\cA_1, \dots, \cA_s \subset \cA$ which are free with
respect to $\phi$. Moreover that for each $1 \leq i \leq s$
we have unital subalgebras $\cB_{i,1}, \dots, \cB_{i, t_i}
\subset \cA_i$ which are free with respect to $\phi$. Then
by \cite[Prop. 2.5.5 (\textit{iii})]{vdn} the subalgebras
$\cB_{1,1}, \dots \cB_{s, t_s} \subset \cA$ are free with
respect to $\phi$.  We shall prove the real second order
version of this. In \cite[Remark 2.7]{mss} the second order
version of \cite{vdn} was left as an exercise for the
reader, now we shall provide a solution. We begin with a
lemma.

\begin{lemma}\label{lemma:linearspoke}
Let $\cA_1, \dots, \cA_s \subset \cA$ be unital subalgebras
which are free with respect to $\phi$. Suppose that $a_1,
\dots, a_m, b_1, \dots, b_n \in \cA$ are such that
\begin{itemize}

\item 
$\phi(a_i) = \phi(b_j) = 0$ for all $i$ and $j$;

\item
$a_i \in \cA_{k_i}$ and $k_1 \not= k_2 \not= \cdots \not= k_m$;

\item
$b_j \in \cA_{l_j}$ and $l_1 \not= l_2 \not= \cdots \not= l_n$.
\end{itemize}
Then for $m \not = n$, $\phi(a_1\cdots a_mb_n \cdots b_1) =
0$ and for $m = n$
\[
\phi(a_1 \cdots a_mb_m\cdots b_1) = \prod_{i=1}^m \phi(a_ib_i).
\]

\end{lemma}

\begin{proof}
Let us begin by showing that
\[
\phi(a_1 \cdots a_m b_n \cdots b_1) =
\phi(a_mb_n) \phi(a_1 \cdots a_{m-1} b_{n-1} \cdots b_1).
\]
First suppose that $k_m \not= l_n$. Then both $\phi(a_1
\cdots a_m b_n \cdots b_1)$ and \ab $\phi(a_mb_n)$ are 0 by
freeness. Thus both sides of the equation above are 0. Next
suppose that $k_m = l_n$ and write $a_mb_n = (a_mb_n)^\circ
+ \phi(a_mb_n)$. Then $\phi(a_1 \cdots a_{m-1}
(a_mb_n)^\circ b_{n-1} \cdots b_1) = 0$ because $k_{m-1}
\not= k_m = l_n \not= l_{n-1}$. Thus
\begin{multline*}
\phi(a_1 \cdots a_nb_1 \cdots b_1) \\
=
\phi(a_1 \cdots a_{m-1} (a_mb_n)^\circ b_{n-1} \cdots b_1) \\
\qquad\mbox{}+
\phi(a_mb_n) \phi(a_1 \cdots a_{m-1} b_{n-1} \cdots b_1) \\
=
\phi(a_mb_n) \phi(a_1 \cdots a_{m-1} b_{n-1} \cdots b_1).
\end{multline*}

Now we conclude by induction. If $m=n$ we get the formula we
claimed. If $m < n$ then
\[
\phi(a_1 \cdots a_m b_n \cdots b_1) =
\phi(a_mb_n) \cdots \phi(a_mb_{n-m+1})
\phi(b_{n-m} \cdots b_1) = 0
\]
by the freeness of the $b_j$'s. The case when $m > n$ is
exactly the same.
\end{proof}

\begin{proposition}\label{proposition:associativity}
Let $\cA_1, \dots, \cA_s \subset \cA$ be $t$-invariant
unital subalgebras of $\cA$ which are real second order free
with respect to $(\phi, \phi_2)$. For each $1 \leq i \leq s$
suppose we have $t$-invariant unital subalgebras $\cB_{i,1}
\dots, \cB_{i, t_i} \subset \cA_i$ which are real free of
second order with respect to $(\phi, \phi_2)$. Then the
subalgebras $\cB_{1,1}, \dots, \cB_{S, t_s} \subset \cA$ are
real free of second order with respect to $(\phi, \phi_2)$.

\end{proposition}

\begin{proof}
The proof of first order freeness is as in \cite[Prop. 2.5.5
  (\textit{iii})]{vdn}. So let us prove part (\textit{ii})
of Definition \ref{definition:secondorderfreeness}. Let
$a_1, \dots, a_m, b_1, \dots, b_n \in \cA$ be such that
\begin{itemize}
\item $\phi(a_i) = \phi(b_j) = 0$ for all $i$ and $j$; and

\item $a_i \in \cB_{k_i, u_i}$ and $(k_1, u_1) \not= (k_2,
  u_2) \not= \cdots \not= (k_m, u_m) \not= (k_1, u_1)$; and

\item
$b_j \in \cB_{l_j, v_j}$ and $(l_1, v_1) \not= (l_2, v_2)
  \not= \cdots \not= (l_n, v_n) \not= (l_1, v_1)$.
\end{itemize}
We must show that for $m = n \geq 2$
\begin{multline}\label{equation:spokesecondorder}
\phi_2(a_1 \cdots a_m, b_1 \cdots b_m) \\ 
= \sum_{\pi \in \spoke{m}} \phi_\pi(a_1 \cdots a_m, b_1
\cdots b_m)\\ +\sum_{\pi \in \rspoke{m}} \phi_\pi(a_1 \cdots
a_m, b_1^t \cdots b_m^t)
\end{multline}
and is 0 for $m \not= n$; the case $m = n = 1$ is immediate. 

Note that adjacent $a_i$'s are, by assumption, from
different $\cB_{k,v}$'s but might be from the same
$\cA_d$. So we group the $a_i$'s according to which $\cA_d$
contains them.  Let $m_1, \dots, m_p$ be positive integers
such that $m_1 + \cdots + m_p = m$ and $a_{m_1 + \cdots +
  m_{i-1}+1}, \dots , a_{m_1 + \cdots + m_i} \in \cA_{d_i}$,
for $1 \leq i \leq p$ and $d_1 \not= d_2 \not= \cdots \not=
d_p \not= d_1$. Then we let $A_i = a_{m_1 + \cdots +
  m_{i-1}+1} \cdots a_{m_1 + \cdots + m_i} \in
\cA_{d_i}$. Then $a_1 \cdots a_m = A_1 \cdots A_p$.

We do exactly the same for the $b_j$'s. Namely we let $n_1,
\dots, n_q$ be positive integers such that $n_1 + \cdots +
n_q = n$ and $b_{n_1 + \cdots n_{j-1}+1}, \dots , b_{n_1 +
  \cdots + n_j} \in \cA_{e_j}$ for $1 \leq j \leq q$ and
$e_1 \not= e_2 \not= \cdots \not= e_q \not= e_1$. We let
$B_j = b_{n+1 + \cdots + b_{j-1}+1} \cdots b_{n_1 + \cdots
  n_j} \in \cA_{e_i}$. Then $b_1 \cdots b_n = B_1 \cdots
B_q$.

Note that by first order freeness 
\[
\phi(A_i) = \phi(a_{m_1 + \cdots + m_{i-1}+1} \cdots a_{m_1
  + \cdots + m_i}) = 0,
\]
since the $\cB_{i,j}$'s are first order free by
\cite[Prop. 2.5.5 (\textit{iii})]{vdn}. Likewise $\phi(B_j)
= 0$.

If $p = q = 1$ then we have
(\ref{equation:spokesecondorder}) by the assumed second
order freeness of $\cB_{1,1}, \dots , \cB_{1,t_1}$. If $p
\not= q$, then by the assumed second order freeness of
$\cA_1, \dots, \cA_s$ we have $\phi_2(a_1\cdots a_m,
b_1\cdots b_n) = 0$, thus the left hand side of
(\ref{equation:spokesecondorder}) is 0.

Let us consider the right hand side of
(\ref{equation:spokesecondorder}). If $m \not= n$ then the
right hand side is 0. So let us suppose that $m = n$. Let us
first consider the term involving $\spoke{m}$. For $\pi \in
\spoke{m}$ and $\phi_\pi(a_1, \dots, a_m, b_1, \dots, b_n)\ab
\not = 0$ we must have $(k_i,i_i) = (l_j, v_j)$ for all
$(i,j) \in \pi$. This means $\pi$ gives a bijection between
the $\cA_i$'s which contain the $a_i$'s and the $\cA_j$'s
which contain the $b_j$'s. So in particular $p = q$, which
is impossible. Likewise if $\pi \in \rspoke{m}$ then we have
a bijection between the $\cA_i$'s containing the $a_i$'s and
the $\cA_j$'s containing the $b_j^t$'s. So again we would
have $p = q$.

Now let us suppose that $p = q \geq 2$. By the assumed real
second order freeness of $\cA_1, \dots, \cA_s$ we have
\begin{multline}\label{equation:coarsesecond}
\phi(a_1 \cdots a_m, b_1 \cdots b_n) 
=
\phi_2(A_1 \cdots A_p, B_1 \cdots B_p)\\
=
\sum_{\pi\in\spoke{p}} \phi_\pi(A_1, \dots, A_p, B_1, \dots, B_p) \\
+
\sum_{\pi\in\rspoke{p}} \phi_\pi(A_1, \dots, A_p, B_1^t, \dots, B_p^t).
\end{multline}
For $\pi \in \spoke{p}$ and $(i, j) \in \pi$ we have by
Lemma \ref{lemma:linearspoke}, when $m_i = n_j$
\[
\phi(A_iB_j) 
=
\phi(a_{m_1 + \cdots m_{i-1}+1} b_{n_1 + \cdots + n_j}) \cdots
\phi(a_{m_1 + \cdots m_i} b_{n_1 + \cdots + n_{j-1}+1})
\]
and 0 when $m_i \not = n_j$. Thus for this $\pi$, assuming
$m_i = n_j$ for all $(i, j) \in \pi$, we have
\[
\phi_\pi(A_1, \dots, A_p, B_1, \dots B_p) =
\phi_{\tilde\pi}(a_1, \dots, a_m, b_1, \dots, b_m)
\]
where $\tilde\pi \in \spoke{m}$ is the spoke diagram
obtained by matching up $m_1 + \cdots + m_{i-1} + k$ with
$n_1 + \cdots n_j - k+1$.

For $\pi \in \rspoke{p}$ and $(i, j) \in \pi$ we have by
Lemma \ref{lemma:linearspoke}, when $m_i = n_j$
\[
\phi(A_iB_j^t) 
=
\phi(a_{m_1 + \cdots m_{i-1}+1} b_{n_1 + \cdots + n_{j-1}+1}^t) \cdots
\phi(a_{m_1 + \cdots m_i} b_{n_1 + \cdots + n_j})
\]
and 0 when $m_i \not = n_j$.  Thus for this $\pi$, assuming
$m_i = n_j$ for all $(i, j) \in \pi$, we have
\[
\phi_\pi(A_1, \dots, A_p, B_1^t, \dots B_p^t) =
\phi_{\tilde\pi}(a_1, \dots, a_m, b_1^t, \dots, b_m^t)
\]
where $\tilde\pi \in \spoke{m}$ is the spoke diagram
obtained by matching up $m_1 + \cdots + m_{i-1} + k$ with
$n_1 + \cdots n_{j-1} + k$.

If we let $\pi$ run over $\spoke{m}$ on the right hand side
of (\ref{equation:coarsesecond}), the corresponding
$\tilde\pi$'s will not exhaust all $\pi$'s on the right hand
side of (\ref{equation:spokesecondorder}), but the ones that
are missed are such that $\phi_\pi(a_1, \dots, a_m, b_1,
\dots, b_m) = 0$, by the first order freeness of the
$\cA_i$'s. Similarly for the $\pi$'s in $\rspoke{m}$ on the
right hand side of (\ref{equation:coarsesecond}). We thus
have
\begin{multline*}
\sum_{\pi\in\spoke{p}} \phi_\pi(A_1, \dots, A_p, B_1, \dots, B_p) \\
+
\sum_{\pi\in\rspoke{p}} \phi_\pi(A_1, \dots, A_p, B_1^t, \dots, B_p^t) \\
=
\sum_{\pi \in \spoke{m}} \phi_\pi(a_1 \cdots a_m, b_1 \cdots b_m) \\
+\sum_{\pi \in \rspoke{m}} 
\phi_\pi(a_1 \cdots a_m, b_1^t \cdots b_m^t).
\end{multline*}
This combined with (\ref{equation:coarsesecond}) proves
(\ref{equation:spokesecondorder}).
\end{proof}

\begin{definition}\label{definition:secondorderlimiting}
Suppose for each $d$ we have random matrices $\{A_{d,1},
\dots,\ab A_{d,s}\}$. We say that the ensemble has a
\textit{real second order limit distribution} if there is a
real second order non-commutative probability space $(\cA,
\phi, \phi_2, t)$ and $a_1, \dots, a_s \in \cA$ such that
for all polynomials $p_1, p_2, p_3, \dots $ in the
non-commuting variables $\{x_1, \dots, x_s, x_1^t, \dots,
x_s^t\}$ we have
\begin{enumerate}
\item $
\lim_{d \rightarrow \infty} \E(\tr(p_1(A_{d,1}, \dots,
  A_{d,s})))
  = \phi(p_1(a_1,  \dots, a_s))$; 
  
\item  \begin{minipage}{4in}
\begin{multline*}
\ds \lim_{d \rightarrow \infty} \cov( \Tr(p_1(A_{d,1}, 
  \dots, A_{d,s})), \\
  \Tr(p_2(A_{d,1}, \dots, A_{d, s})))\\
  = 
  \phi_2( p_1(a_1, \dots, a_s), p_2(a_1, \dots, a_s)) 
  \end{multline*}\end{minipage}
 
\item
for all $r \geq 3$

\[
\lim_{d\rightarrow \infty} \rk_r(\Tr(p_1(A_{d,1},
\dots, A_{d,s})), \dots, 
\Tr(p_r(A_{d,1}, A_{d,1}^t, \dots, A_{d,s}))) = 0.
\]
\end{enumerate}

\end{definition}

\begin{remark}
The third condition is only needed to ensure the convergence
of fluctuations of mixed moments. In fact boundedness would
be enough. For many ensembles of matrices the $r^{th}$
cumulant vanishes on the order of $d^{2-r}$, for example the
ensembles discussed in \cite{r1,r2}. For deterministic
matrices the higher cumulants of traces are 0. Moreover a
close reading of our proof shows that if one starts with an
ensemble $\{A_i\}_i$with $\rk_r$ between $\lo(1)$ and
$\rO(1)$ for $r \geq 3$, the mixed cumulants of $A$'s and
$O$'s for $r \geq 3$ would have the same order as
$\{A_i\}_i$.

\end{remark}

\begin{remark}
Suppose we have for each $d$, random matrices $\{A_{d,1},\ab
\dots,\ab A_{d,s}\}$, a non-commutative probability space
$(\cA, \phi)$, and $a_1, \dots, a_n \in \cA$ such that for
every polynomial $p$ in the non-commuting variables $x_1,
\dots, x_s, x_1^t, \dots, x_s^t$ we have
\[
\lim_{d \rightarrow \infty} \tr(p(A_{d,1}, \dots, A_{d, n}))
= \phi(p(a_1, \dots, a_n))
\]
then we say that the matrices $\{A_{d,1}, \dots, A_{d,n}\}$
have the \textit{limit joint $t$-distribution} given by
$a_1, \dots, a_n$.
\end{remark}

\begin{definition}\label{definition:asymptoticrealsecond}
Let $\{A_{d,1}, \dots, A_{d,r}\}_d$ and $\{B_{d,1}. \dots,
B_{d,s}\}_d$ be two ensembles of random matrices such that
$\{A_{d,1}, \dots, A_{d,r}, B_{d,1}. \dots, B_{d,s}\}_d$ has
a real second order limit distribution given by $\{a_1,
\dots, a_s, b_1, \dots,\ab b_s\}$ in the real second order
non-commutative probability space $(\cA, \phi,\ab \phi_2,\ab
t)$. If the two unital subalgebras $\cA_1 = \alg(1, a_1,
\dots, a_r, a_1^t, \dots, a_r^t)$ and $\cA_2 = \alg(1, b_1,
\dots, b_s, b_1^t, \dots, b_s^t)$ are real free of second
order then we say that the two ensembles $\{A_{d,1}, \dots,
A_{d,r}\}_d$ and $\{B_{d,1}. \dots, B_{d,s}\}_d$ are
\textit{asymptotically real free of second order}.

\end{definition}


\section{First Order Freeness of Haar Orthogonal \\ 
         and Independent Matrices}\label{section:firstorderfreenes}

To show that a family of $d \times d$ random matrices $\{
A_{1}, \dots, A_{s}\}_d$ and an independent family of
orthogonal matrices $\{O_d\}_d$ are asymptotically real free
of second order, we must first demonstrate that they are
asymptotically free of first order, or asymptotically free
in the sense of Voiculescu \cite[\S 2.5]{vdn}.

For this we must show that given polynomials $\{p_1,\ab
\dots, \ab p_n\}$ in $O$ and $O^{-1}$ such that
$\E(\tr(p_i(O, O^{-1}))) = 0$ and random matrices
$\{A_{1}, \dots,\ab A_{s}\}$ with $\E(\tr(A_{i})) =
0$, then
\[
\lim_{d \rightarrow \infty}
\E(\tr(p_1(O, O^{-1}) A_1 \cdots p_n(O, O^{-1}) A_s))
= 0
\]
provided that the entries of the $A_{d,i}$'s are independent
from those of the $O$'s and the $\{A_{d,1}, \dots,
A_{d,n}\}$ have a real second order limit distribution. For
this it suffices to prove that
\[
\lim_{d \rightarrow \infty}
\E(\tr(O^{m_1} A_{1} \cdots O^{m_n} A_{s}))
= 0
\]
for any sequence of non-zero integers $m_1, \dots, m_n$ and
$\{A_{1}, \dots,\ab A_{s}\}$ as above.

\begin{notation}\label{notation:cumulantoftrace}
Let $\pi \in S_n$ be a permutation and $\cU \in \cP(n)$ be a
partition such that each cycle of $\pi$ lies in some block
of $\cU$. We denote this relation by $\pi \leq \cU$. Let
$A_1, \dots, A_n$ be $d \times d$ random matrices and write,
as in equation (\ref{equation:permutation}),
\[
\Tr_\pi(A_1, \dots, A_n) = \sum_{i_1, \dots, i_n}^d
a^{(1)}_{i_1i_{\pi(1)}} \cdots a^{(n)}_{i_ni_{\pi(n)}}
\]

Let the blocks of $\cU$ be $\{U_1, \dots, U_k\}$ and let
$\pi_i$ be the product of cycles of $\pi$ that lie in
$U_i$. If $c = (i_1, \dots, i_r)$ is a cycle of $\pi$, let
$\Tr_c(A_1, \dots, A_n) = \Tr(A_{i_1} \cdots A_{i_r})$. If
$\pi_i = c_1 \cdots c_k$, as a product of cycles, let
$\Tr_{\pi_i}(A_1, \dots, A_n) = \prod_i \Tr_{c_i}(A_1,
\dots, A_n)$. Next let
\begin{equation}\label{equation:partitionnotation}
\E_\cU(\Tr_\pi(A_1, \dots, A_n)) 
=
\prod_{i=1}^k \E(\Tr_{\pi_i}(A_1, \dots, A_n)).
\end{equation} 
Finally for $\eta = (\eta_1, \eta_2, \dots, \eta_n) \in
\bZ_2^n$ and $\pi \in S_n$, let
\[
\E_\cU(\Tr_{(\pi, \eta)}(A_1, \dots, A_n)) =
\E_\cU(\Tr_\pi(A^{(\eta_1)}_1, \dots, A^{(\eta_n)}_n)).
\]
To make this clear let us give an example. Let $n = 6$, $\pi
= (1)(2,4)(3)$ and $\cU = \{(1,3),(2,4)\}$. Then
\[
\E_\cU(\Tr_\pi(A_1, A_2, A_3, A_4)) =
\E(\Tr(A_1)\Tr(A_3))\E(\Tr(A_2A_4)).
\]
We shall also need to work with the normalized trace $\tr =
d^{-1}\Tr$. We let $\tr_\pi(A_1, \dots, A_n) = d^{-\#(\pi)}
\Tr_\pi(A_1, \dots, A_n)$.

If $\cU \in \cP(n)$ and $\pi \leq \cU$, in the sense above,
then we let
\begin{equation}\label{equation:cumulantnotation}
\rk_\cU(\Tr_{\pi}(A_1, \dots, A_n))
=
\mathop{\sum_{\cV \in \cP(n)}}_{\pi \leq \cV \leq \cU}
m(\cV, \cU) \E_\cV(\Tr_{\pi_i}(A_1, \dots, A_n)).
\end{equation}
Then by M\"obius inversion we have
\begin{equation}\label{equation:moebiusinversion}
\E_\cU(\Tr_{\pi_i}(A_1, \dots, A_n))
=
\mathop{\sum_{\cV \in \cP(n)}}_{\pi \leq \cV \leq \cU}
\rk_\cV(\Tr_{\pi}(A_1, \dots, A_n)).
\end{equation}
\end{notation}

\begin{remark}
In what follows, for an ensemble of $d \times d$ matrices $\{A_{1},
\dots,\ab A_{s}\}_d$, will suppress the dependency of
$A_{i}$ on $d$ and just denote it by $A_i$. Moreover the
$(i,j)$-entry of $A_k$ will be denoted $a^{(k)}_{ij}$. This
should not cause any confusion as at each stage of the
discussion we shall only be multiplying matrices of the same
size. Likewise for an ensemble of random orthogonal
orthogonal matrices $\{O_d \}_d$, we shall drop the
dependence on $d$ from the notation.
\end{remark}

\begin{theorem}\label{theorem:firstorderfreeness}
Let for each $d$, $\{ A_{1}, \dots, A_{n}\}$ be a
ensemble of centred $d \times d$ random matrices that have a
real second order limit distribution, $O$ a Haar distributed
random $d \times d$ orthogonal matrix, and $m_1, \dots, m_n$
be non-zero integers. Then
\[
\lim_{d \rightarrow \infty}
\E(\tr(O^{m_1} A_{1} \cdots O^{m_n} A_{n}))
= 0.
\]
\end{theorem}

\begin{proof}
In order to be able to use the result of Proposition
\ref{proposition:exact_formula}, with $\gamma= (1, 2,
3. \dots, n)$, we have to reduce it to the case of each
$m_i$ being either $1$ or $-1$. We can achieve this by
inserting an identity matrix, $I$, between any two adjacent
$O$'s or adjacent $O^{-1}$'s. For example $O^2A_1O^{-1}A_2$
would become $O I O A_1 O^{-1} A_2$. So with this change we
must show that, whenever we have $\epsilon_1, \dots,
\epsilon_n \in \{-1, 1\}$ and random matrices $A_1, \dots,
A_n$ with a limit joint $t$-distribution such that for
each $i$, either $A_i$ is centred, i.e. $\E(\tr(A_i)) = 0$,
or $A_i = I$ and $\epsilon_{i} = \epsilon_{\gamma(i)}$, then
\[
\lim_{d \rightarrow \infty}
\E(\tr(O^{\epsilon_1} A_1 \cdots O^{\epsilon_n} A_n))
= 0.
\]

By Proposition \ref{proposition:exact_formula}
\begin{eqnarray*}
\lefteqn{
\E(\Tr(O^{\epsilon_1}A_1, \dots, O^{\epsilon_n}A_n))}\\ & =
  & \sum_{p,q \in \cP_2(n)} \langle \Wg(p), q\rangle
  \E(\Tr_{(\pi_{p \mskip 2mu\cdot_\epsilon q}, \eta_{p \mskip
      2mu\cdot_\epsilon q})}(A_1, \dots, A_n)).
\end{eqnarray*}

Let us recall the construction of $\pi_{p \mskip
  2mu\cdot_\epsilon q}$. We write the permutation $(p \mskip
2mu\cdot_\epsilon q) \delta$, which is the product of two
pairings, as a product of cycles. We showed that the cycles
always occur in pairs of the form $\{c, c'\}$, where $c' =
\delta c^{-1} \delta$. From each pair we choose one, and
then from this we obtained a cycle of $\pi_{p \mskip
  2mu\cdot_\epsilon q}$ by deleting any minus signs. The
minus signs that are deleted are recorded in $\eta_\peq$. So
let us consider the singletons of $\pi_{p \mskip
  2mu\cdot_\epsilon q}$. If $(k)$ is a singleton of
$\pi_{p \mskip 2mu\cdot_\epsilon q}$, then $(p \mskip
2mu\cdot_\epsilon q) \delta$ will have the two singletons
$(k)(-k)$ and thus $(k, -k)$ will be a cycle of $(p \mskip
2mu\cdot_\epsilon q)$ and hence $(-\delta_\epsilon(k),
\delta_\epsilon(\gamma(k)))$ will be a cycle of $p \delta q
\delta$. The cycles of $p \delta q \delta$ are either cycles
of $p$, consisting of pairs of positive numbers, or cycles
of $\delta q \delta$, consisting of pairs of negative
numbers. \textit{Thus if $(k)$ is a singleton of
  $\pi_{p \mskip 2mu\cdot_\epsilon q}$ then we must have
  $\epsilon_k = -\epsilon_{\gamma(k)}$, and hence $A_k$ is a
  centred matrix.}

Now consider the expansion
\begin{eqnarray*}
\lefteqn{
\E(\Tr(O^{\epsilon_1}A_1, \dots, O^{\epsilon_n}A_n))}\\ & =
  & \sum_{p,q \in \cP_2(n)} \langle \Wg(p), q\rangle
\E(\Tr_{(\pi_\peq, \eta_\peq)}(A_1, \dots, A_n)).
\end{eqnarray*}
We have
\[
\langle \Wg(p), q \rangle = \rO(d^{-n + \#(p \vee q)}).
\]
We must next find a upper bound for the order of
\[
\E(\Tr_{(\pi_{p \mskip 2mu\cdot_\epsilon q}, \eta_{p \mskip
      2mu\cdot\epsilon q})}(A_1, \dots, A_n) )
= \sum_{\pi \leq \cU}
\rk_\cU(\Tr_{(\pi_\peq, \eta_\peq)}(A_1, \dots, A_n)).
\] 
Since $\{A_1, \dots, A_n\}$ has a real second order limit
distribution we have that
\[
\rk_\cU(\Tr_{(\pi_\peq, \eta_\peq)}(A_1, \dots, A_n))
=
\rO(d^u)
\]
where $u$ is the number of blocks of $\cU$ that contain a
single cycle of $\pi$. If $\cU$ has a singleton $(k)$ then
$\pi$, too, will have a singleton $(k)$ and then $A_k$ will
be centred so $\rk_\cU(\Tr_{(\pi_\peq, \eta_\peq)}(A_1,
\dots, A_n))$ will have a factor $\E(\Tr(A_k))= 0$, hence
$\rk_\cU(\Tr_{(\pi_\peq, \eta_\peq)}(A_1, \dots, A_n)) = 0$.

Thus $u \leq \#(\cU) \leq n/2$ and so $-n + \#(p \vee q) + u
\leq 0$, thus
\[
\langle \Wg(p), q\rangle
\E(\Tr_{(\pi_\peq, \eta_\peq)}(A_1, \dots, A_n))
=
\rO(1).
\]

Thus
\[
\E(\Tr(O^{\epsilon_1}A_1, \dots, O^{\epsilon_n}A_n))
      = \rO(1)
\]
and hence
\[
\lim_{d \rightarrow \infty}
\E(\tr(O^{\epsilon_1}A_1, \dots, O^{\epsilon_n}A_n))
      = 0.
\]
\end{proof}

\begin{corollary}\label{corollary:firstcumulants}
Let $\{ A_1, \dots, A_{n+1}\}$ be $d \times d$ random
matrices whose entries have moments of all orders, $O$ a
Haar distributed random $d \times d$ orthogonal matrix,
independent from $\{ A_1, \dots, A_{n+1}\}$, and
$\epsilon_1, \dots, \epsilon_n \in \bZ_2$. Suppose that for
each $1 \leq i \leq n$ we have that either $\E(\Tr(A_i)) =
0$ \textit{or} $A_i = I$ and $\epsilon_i = \epsilon_{i+1}$
(using $\epsilon_{n+1} = \epsilon_1$), and $\E(\Tr(A_{n+1}))
= 0$. Then
\[
\E(\Tr(O^{\epsilon_1} A_1 \cdots O^{\epsilon_n} A_n))
= \rO(1),
\]
in fact
\begin{eqnarray}\label{equation:firstcumulantandp}\lefteqn{
\E(\Tr(O^{\epsilon_1} A_1 \cdots O^{\epsilon_n} A_n))}\notag\\
& = &
d^{-n/2}\sum_{p \in \cP_2(n)} \E_{\pi_\pep}(
\Tr_{(\pi_\pep, \eta_\pep)}(A_1, \dots, A_n)) + \rO(d^{-1})
\end{eqnarray}
where the sum is over all $p$'s such that $\pi_\pep$ is a
pairing and
\[
\E(\Tr(O^{\epsilon_1} A_1 \cdots O^{\epsilon_n} A_n)\Tr(A_{n+1}))
= \rO(d^{-1}).
\]
\end{corollary}

\begin{proof}
The first claim is just the second last equation of the
proof of Theorem \ref{theorem:firstorderfreeness}. Recall
that when we expand into cumulants
\[
\E(\Tr_{(\pi_\peq, \eta_\peq)}(A_1, \dots, A_n))
=
\mathop{\sum_{\cU \in \cP(n)}}_{\pi_\peq \leq \cU}
 \rk_\cU(A_1, \dots, A_n)
\] 
and let $u$ be the number of blocks of $\cU$ that contain a
single cycle of $\pi_\peq$ we have $-n + \#(p \vee q) + u
\leq 0$ with equality only when $p = q$ and $n = n/2$,
i.e. $\cU = \pi_\pep$ and $\pi_\pep$ is a pairing. This
establishes the second claim.

By Proposition \ref{proposition:exact_formulafirstextension}
we have
\[
\E(\Tr(O^{\epsilon_1} A_1 \cdots O^{\epsilon_n} A_n)\Tr(A_{n+1}))
\]
\[ =
\sum_{p,q \in \cP_2(n)} \kern-1em \langle
  \Wg(p), q\rangle \E(\Tr_{(\pi_\peq, \eta_\peq)}(A_1, \dots,
  A_n) \Tr(A_{n+1})).
  \]
For the moment let us fix $p,q \in \cP_2(n)$ and let
$\tilde\pi \in S_{n+1}$ be the permutation which fixes $n+1$
and whose restriction to $[n]$ is $\pi_\peq$. Likewise let
$\tilde\eta|_{[n]} = \eta_\peq$ and $\tilde\eta_{n+1}=
1$. Then $\E(\Tr_{(\pi_\peq, \eta_\peq)}(A_1, \dots,
A_n)\Tr(A_{n+1}))\ab = \E(\Tr_{(\tilde\pi, \tilde\eta)}(A_1,
\dots, A_{n+1}))$ Then we expand as above
\[
\E(\Tr_{(\tilde\pi, \tilde\eta)}(A_1, \dots, A_{n+1}))  =
\mathop{\sum_{\cU \in \cP(n+1)}}_{\tilde\pi \leq \cU}
\rk_\cU(\Tr_{(\tilde\pi, \tilde\eta)}(A_1, \dots, A_{n+1})).
\]
Suppose $\cU \in \cP(n+1)$ is such that $\tilde\pi \leq \cU$
and $\rk_\cU(\Tr_{(\tilde\pi, \tilde\eta)}(A_1, \dots,\ab
A_{n+1})) \not = 0$. Then
\[
\rk_\cU(\Tr_{(\tilde\pi, \tilde\eta)}(A_1, \dots,\ab
A_{n+1})) = \rO(d^u)
\]
where $u$ is the number of blocks of $\cU$ that contain only
one cycle of $\tilde\pi$. Since, by assumption,
$\E(\Tr(A_{n+1})) = 0$, the last cycle of $\tilde\pi$ cannot
be in a block of $\cU$ on its own (otherwise $\rk_\cU= 0$);
thus $u \leq \#(\cU) - 1$. As in the proof of Theorem
\ref{theorem:firstorderfreeness}, $\#(\cU|_{[n]}) \leq n/2$
and as the cycle $(n+1)$ cannot be on its own we have
$\#(\cU) \leq n/2$. So $u \leq n/2 -1$.  Thus $-n + \#(p
\vee q) + u \leq -1$ and so
\[
\langle \Wg(p), q\rangle \rk_\cU(\Tr_{(\tilde\pi,
  \tilde\eta)}(A_1, \dots,\ab A_{n+1})) = \rO(d^{-1}).
\]
Since this holds for every $\cU$ we have 
\[
\langle \Wg(p), q\rangle \E(\Tr_{(\tilde\pi,
  \tilde\eta)}(A_1, \dots,\ab A_{n+1})) = \rO(d^{-1}).
\]
Since this in turn holds for every $p$ and $q$ we have
\[
\E(\Tr(O^{\epsilon_1} A_1 \cdots O^{\epsilon_n}
A_n)\Tr(A_{n+1})) = \rO(d^{-1}).
\]
\end{proof}


\section{Fluctuation Moments of Haar Orthogonal\\
 and Independent Random Matrices}\label{section:fluctuationmoments}

Our next step is to show that the limit distribution of Haar
distributed orthogonal matrices and an independent ensemble of random
matrices with a real second order limit distribution
satisfies part (\textit{ii})\, (\textit{b}) of Definition
\ref{definition:secondorderfreeness}. Fix positive integers
$m$ and $n$ and let $\gamma$ be the permutation with the two
cycles $(1, \dots, m)(m+1, \dots, m+n)$.

\begin{theorem}\label{theorem:secondorderlimit}
Let $\{A_{1},\dots, A_{m}\}$ and $\{B_{1}, \dots,
B_{n}\}$ be a ensemble of centred $d \times d$ matrices
that have a real second order limit distributions given by
$(a_1, \dots, a_m)$ and $(b_1, \dots, b_n)$, respectively,
in a real second order non-commutative probability space
$(\cA, \phi, \phi_2, t)$, and $O$ a Haar distributed
random $d \times d$ orthogonal matrix, and $k_1, \dots,
k_{m},l_1, \dots, l_n$ non-zero integers. Suppose that the
entries of $\{A_{1}, \dots, A_{m}, B_{1}, \dots,\ab
B_{n}\}$ are independent from those of $O$. Then
\[
\lim_{d \rightarrow \infty}
\cov( \Tr(O^{k_1}A_{1} \cdots O^{k_m} A_{m}),
      \Tr(O^{l_{1}} B_{1} \cdots O^{l_{n}} B_{n})) 
\]
exists and equals 0 when $m \not= n$, and when $m = n \geq 2$,
equals
\begin{equation}\label{equation:secondorderside}   
\sum_{r=1}^m \bigg\{
\prod_{i=1}^m \phi(a_i b_{r-i})\phi(o^{k_i + l_{r-(i-1)}})
+
\prod_{i=1}^m \phi(a_i b_{r+i}^t) \phi(o^{k_i - l_{r+i}})
\bigg\}.
\end{equation}
where the indices of the $b$'s and $l$'s are taken modulo
$m$.
\end{theorem}

\begin{proof} 
We begin
by noting that by Theorem \ref{theorem:weingartenexpansion},
$m +n $ must be even, otherwise the limit of the covariances
is 0. In order to apply Proposition
\ref{proposition:exact_formula} to the expression
\[
\cov(\Tr(O^{k_1} A_1 \cdots O^{k_m}A_m), \Tr(O^{l_1}B_1
\cdots O^{l_n}B_n))
\]
we have to reduce it to the case where all $k$'s and $l$'s
are either 1 or $-1$. So let us consider the term
$\phi(a_ib_{r-i}) \phi(o^{k_i + l_{r-(i-1)}})$ of expression
(\ref{equation:secondorderside}). In order for this to be
non-zero we must have $k_i + l_{r-(i-1)} = 0$. So when we
perform the reduction used in the proof of Theorem
\ref{theorem:firstorderfreeness} we replace $o^{k_i}$,
supposing $k_i > 0$, with $o1o\cdots o1o$ and
$o^{l_{r-i+1}}$ with $o^{-1}1o^{-1} \cdots
o^{-1}1o^{-1}$ the factor $\phi(o^{k_i + l_{r-i+1}}) = 1$
gets replaced by $\phi(oo^{-1}) \phi(11) \phi(oo^{-1})
\cdots \phi(oo^{-1}) \phi(11) \phi(oo^{-1}) = 1$. Likewise 
with the factor $\phi(o^{k_i - l_{r+i}})$. Thus without loss of
generality we can assume that $k_1\dots, k_m l_1, \dots, l_n
\in \{-1, 1\}$. In this case we must show that
\[
\lim_{d \rightarrow \infty} \cov( \Tr(O^{\epsilon_1}A_{1}
\cdots O^{\epsilon_m} A_{m}),
\Tr(O^{\epsilon_{m+1}}B_{1} \cdots O^{\epsilon_{m+n}}
B_{n}))
\]
exists and equals 0 when $m \not = n$ and when $m = n$
equals
\begin{equation}\label{equation:secondorderlimit}
\sum_{r=1}^m \bigg\{ \prod_{i=1}^m \phi(a_i
b_{r-i})\delta_{\epsilon_i, -\epsilon_{\gamma^{-i+1}(m + r)}} +
\prod_{i=1}^m \phi(a_i b_{r+i}^t) \delta_{\epsilon_i,
  \epsilon_{\gamma^{i}(m+r)}} \bigg\},
\end{equation}
where the $\gamma$ in the index of the second $\epsilon$ in
$\delta_{\epsilon_i, \epsilon_{\gamma^{-i+1}(m+r)}}$ is the
permutation with cycle decomposition $(1, \dots, m)(m+1, \dots, 2m)$.

By Proposition
\ref{proposition:exact_formula}
\[
\E( \Tr(O^{\epsilon_1} A_1 \cdots O^{\epsilon_m}  A_m) 
\Tr(O^{\epsilon_{m+1}} B_{1}\cdots O^{\epsilon_{m+n}} B_{n}))
\]
\[
=
\sum_{p, q \in \cP_2(m+n)} \langle \Wg(p), q \rangle
\E(\Tr_{(\pi_\peq, \eta_\peq)}(A_1, \dots, B_n)),
\]
and
\[
\E( \Tr(O^{\epsilon_1} A_1 \cdots O^{\epsilon_m} A_m)) \,
\E( \Tr(O^{\epsilon_{m+1}} B_{1} \cdots O^{\epsilon_{m+n}} B_{n}))
\]
\[
=
\mathop{\sum_{p, q \in \cP_2(m+n)}}_{p, q \leq \gamma}
\Wg(\gamma, p, q) \E_\gamma(\Tr_{\pi_\peq}(\vec{A}^\eta, \vec{B}^\eta)).
\]
To simplify the notation we let
$\Tr_{\pi_\peq}(\vec{A}^\eta,\ab \vec{B}^\eta) =
\Tr_{(\pi_\peq, \eta_\peq)}(A_1,\ab \dots,\ab B_n)$. Thus
   {\setlength{\arraycolsep}{1pt}
\begin{eqnarray}\lefteqn{
\cov( \Tr(O^{\epsilon_1}A_{1}
\cdots O^{\epsilon_m} A_{m}),
\Tr(O^{\epsilon_{m+1}}B_{1} \cdots O^{\epsilon_{m+n}}
B_{n}))} \notag \\
& = &
\sum_{p, q \in \cP_2(m+n)} \langle \Wg(p), q \rangle
\E(\Tr_{\pi_\peq}(\vec{A}^\eta,\ab \vec{B}^\eta)) \notag  \\
& & \mbox{}-
\mathop{\sum_{p, q \in \cP_2(m+n)}}_{p, q \leq \gamma}
\Wg(\gamma, p, q) 
\E_\gamma(\Tr_{\pi_\peq}(\vec{A}^\eta,\ab \vec{B}^\eta)) \notag \\
& = &
\mathop{\sum_{p, q \in \cP_2(m+n)}}_{p \vee q \vee \gamma = 1_{m+n}}
\langle \Wg(p), q \rangle
\E(\Tr_{\pi_\peq}(\vec{A}^\eta,\ab \vec{B}^\eta)) 
\label{expression:firstterm} \\
& & \mbox{} + \kern-1.0 em
\mathop{\sum_{p, q \in \cP_2(m+n)}}_{p, q \leq \gamma}
\kern-1em
\big\{ \langle \Wg(p), q\rangle - \Wg(\gamma, p, q) \big\}
\E(\Tr_{\pi_\peq}(\vec{A}^\eta,\ab \vec{B}^\eta)) 
\label{expression:secondterm} \\
& & \mbox{} + \kern-1.5em
\mathop{\sum_{p, q \in \cP_2(m+n)}}_{p, q \leq \gamma}
\kern-1.5em \Wg(\gamma, p, q)
\big\{\E(\Tr_{\pi_\peq}(\vec{A}^\eta,\ab \vec{B}^\eta))
- \E_\gamma(\Tr_{\pi_\peq}(\vec{A}^\eta,\ab \vec{B}^\eta))
\big\}.\quad \label{expression:thirdterm}
\end{eqnarray}}

We shall show that the first term
(\ref{expression:firstterm}) converges to
\[
\sum_{r=1}^m \bigg\{ \prod_{i=1}^m \phi(a_i
b_{r-i})\delta_{\epsilon_i, -\epsilon_{\gamma^{-i+1}(m+r)}} +
\prod_{i=1}^m \phi(a_i b_{r+i}^t) \delta_{\epsilon_i,
  \epsilon_{\gamma^{i}(m+r)}} \bigg\},
\]
and the second (\ref{expression:secondterm}) and third term
(\ref{expression:thirdterm}) converge to 0.

We first consider expression (\ref{expression:firstterm}),
and show that this has the limit we have claimed.  Let us
find the order of $\E(\Tr_{\pi_\peq}(\vec{A}^\eta,\ab
\vec{B}^\eta))$; to do this we have to rewrite this
expectation in terms of cumulants so that we can use our
assumptions about the $A$'s and $B$'s having a real second
order limit distribution.  If we consider $\pi_\peq$ a
partition of $[m+n]$ then by equation
(\ref{equation:moebiusinversion}) we have
\begin{eqnarray}\lefteqn{\label{equation:expectationexpansion}
\E(\Tr_{(\pi_\peq, \eta_\peq)}(A_1, \dots, B_n)) } \notag\\
& = & \kern-1em
\mathop{\sum_{\cU \in \cP(m+n)}}_{\cU \geq \pi_\peq}
\rk_\cU(\Tr_{(\pi_\peq, \eta_\peq)}(A_1, \dots, B_n)).
\end{eqnarray}
Suppose $\cU \in \cP(m+n)$ and $\cU \geq \pi_\peq$. If $\cU$
has a singleton $(k)$, then $(k)$ is also a singleton of
$\pi_\peq$. As in the proof of Theorem
\ref{theorem:firstorderfreeness}, this implies that $A_k$
(or $B_{k-m}$ if $k > m$) is centred, and thus,
$\rk_\cU(\Tr_{\pi_\peq}(\vec{A}^\eta, \vec{B}^\eta)) =
0$. Thus we only have to consider $\cU$'s with no
singletons. Hence $\#(\cU) \leq (m+n)/2$.  Suppose $U$ is a
block of $\cU$ which contains two or more cycles of
$\pi_\peq$; the corresponding factor in Equation
(\ref{equation:expectationexpansion}) is a second or higher
cumulant of traces, which converge by our assumption that
the $A$'s and $B$'s have a real second order limit
distribution. Hence these factors will be of order
$\rO(d^0)$. Each block of $\cU$ which contains only one
cycle of $\pi_\peq$ will be of order $\rO(d)$. Hence
$\rk_\cU(\Tr_{\pi_\peq}(\vec{A}^\eta, \vec{B}^\eta)) =
\rO(d^u)$ where $u$ is the number of blocks of $\cU$ which
contain only one cycle of $\pi_\peq$. As
\[
u \leq \#(\cU) \leq (m+n)/2,
\]
we have $\rk_\cU(\Tr_{\pi_\peq}(\vec{A}^\eta, \vec{B}^\eta))
= \rO(d^{(m+n)/2})$ and the order $(m+n)/2$ can only be
achieved when $u = (m+n)/2$, which implies that $\pi_\peq =
\cU$, as partitions, and no cycle of $\pi_\peq$ is a
singleton, because no block of $\cU$ is a singleton. If
$\#(\pi_\peq) = u = (m+n)/2$ and $\pi_\peq$ has no
singletons; $\pi_\peq$ must be a pairing. Combining these
conclusions we have

\[
\E(\Tr_{\pi_\peq}(\vec{A}^\eta, \vec{B}^\eta)) = \rO(d^{(m+n)/2 -1})
\]
unless $p = q$ and $\pi_\peq$ is a pairing, in which case
\begin{eqnarray}\label{equation:secondcumulant}\lefteqn{%
\E(\Tr_{\pi_\peq}(\vec{A}^\eta, \vec{B}^\eta)) } \notag\\ &
  = & \E_{\pi_\peq}(\Tr_{\pi_\peq}(\vec{A}^\eta,
  \vec{B}^\eta)) + \rO(d^{(m+n)/2 -1}).
\end{eqnarray}
Using our usual bound on the order of $\Wg$, namely
\[ 
\langle \Wg(p), q \rangle = \rO(d^{-(m+n) + \#(p \vee q)}),
\] 
we thus have
\[
\langle \Wg(p), q \rangle 
\E(\Tr_{\pi_\peq}(\vec{A}^\eta, \vec{B}^\eta))
= \rO(d^{-1})
\]
unless $p = q$ and $\pi_\peq$ is a pairing, in which case
\begin{eqnarray*}\lefteqn{%
\langle \Wg(p), q \rangle \E(\Tr_{\pi_\peq}(\vec{A}^\eta,
\vec{B}^\eta)) } \\ & = &
  \E_{\pi_\peq}(\tr_{\pi_\peq}(\vec{A}^\eta, \vec{B}^\eta))
  + \rO(d^{-1}).
\end{eqnarray*}

Thus
\begin{eqnarray}\label{equation:highestorderterms}\lefteqn{
\mathop{\sum_{p, q \in \cP_2(m+n)}}_{p \vee q \vee \gamma = 1_{m+n}}
\langle \Wg(p), q \rangle
\E(\Tr_{\pi_\peq}(\vec{A}^\eta,\ab \vec{B}^\eta))}\notag\\
& = &
\sum_{p \in \cP_2(m+n)}
\E_{\pi_\pep}(\tr_{\pi_\pep}(\vec{A}^\eta, \vec{B}^\eta)) + \rO(d^{-1})
\end{eqnarray}
where the second sum runs over all $p$ such that $p \vee
\gamma = 1_{m+n}$ and $\pi_\pep$ is a pairing. To find the
limit as $d \rightarrow \infty$ we use Lemmas
\ref{lemma:standardspokediagram} and
\ref{lemma:reversedspokediagram}.

First suppose that there is $(u, v) \in p$ such that
$\epsilon_u = -\epsilon_v$. Then by Lemma
\ref{lemma:standardspokediagram} we have $m = n$, every
cycle of $p$ connects the two cycles of $\gamma$, and
$\epsilon_i = -\epsilon_j$ for all $(i, j) \in p$. Then for
some $r \in [m]$ we have $(m-1, m+r) \in p$. Again by Lemma
\ref{lemma:standardspokediagram} we have for all $k \in [m]$

\begin{itemize}

\item $(k, \gamma^{-k}(m+r)) \in \pi_\pep$,

\item $(k, \gamma^{-k+1}(m+r)) \in p$,

\item $\eta_k = 1$.
\end{itemize}
Thus $\epsilon_k = -\epsilon_{\gamma^{-k+1}(m + r)}$ and
\begin{eqnarray}\label{equation:intermediatestraightstep}
\E_{\pi_\pep}(\tr_{(\pi_\pep,\eta_\pep)}(A_1, \dots, B_m)) =
\prod_{k=1}^m \E(\tr(A_kB_{r-k})) \delta_{\epsilon_k,
  -\epsilon_{\gamma^{-i+1}(m + r)}}
\end{eqnarray}
which converges to
\[
\prod_{k=1}^m \phi(a_kb_{r-k}) \delta_{\epsilon_k,
  -\epsilon_{\gamma^{-k+1}(m + r)}}
\]
as $d \rightarrow \infty$.

Next suppose that there is $(u, v) \in p$ such that
$\epsilon_u = \epsilon_v$. Then by Lemma
\ref{lemma:reversedspokediagram} we have $m = n$, every
cycle of $p$ connects the two cycles of $\gamma$, and
$\epsilon_i = \epsilon_j$ for all $(i, j) \in p$. Then for
some $r \in [m]$ we have $(m-1, m+r) \in p$. As in Lemma
\ref{lemma:reversedspokediagram}, let $l =
\gamma^{-m+1}(m+r) = \gamma^r(2m)$. Then $\gamma^k(l) =
\gamma^r(m+k)$, for $k \in [m]$. Hence by Lemma
\ref{lemma:reversedspokediagram} we have for all $k \in [m]$

\begin{itemize}

\item $(k, \gamma^{r}(m+k)) \in \pi_\pep$,

\item $(k, \gamma^{r}(m+k)) \in p$,

\item $\eta_k = -1$.
\end{itemize}
Thus $\epsilon_k = \epsilon_{\gamma^{k}(m + r)}$. 
 
\begin{eqnarray}\label{equation:intermediatereversedstep}
\E_{\pi_\pep}(\tr_{(\pi_\pep,\eta_\pep)}(A_1, \dots, B_m)) =
\prod_{k=1}^m \E(\tr(A_kB_{r+k}^t)) \delta_{\epsilon_k,
  \epsilon_{\gamma^{k}(m + r)}}
\end{eqnarray}
which converges to
\[
\prod_{k=1}^m \phi(a_kb_{r+k}) \delta_{\epsilon_k,
  \epsilon_{\gamma^{k}(m + r)}}
\]
as $d \rightarrow \infty$.
Hence the expression (\ref{expression:firstterm}) converges to
\[
\sum_{r=1}^m \bigg\{ \prod_{k=1}^m \phi(a_k
b_{r-k})\delta_{\epsilon_k, -\epsilon_{\gamma^{-i+1}(m+ r)}} +
\prod_{k=1}^m \phi(a_k b_{r+k}^t) \delta_{\epsilon_k,
  \epsilon_{\gamma^{i}(m+r)}} \bigg\}.
\]

To show that (\ref{expression:secondterm}) and
(\ref{expression:thirdterm}) vanish as $d \rightarrow
\infty$ we have to consider the order of
$\E_\gamma(\Tr_{\pi_\peq}(\vec{A}^\eta, \vec{B}^\eta))$ with
$p,q\leq \gamma$. As before we write this as a sum of
cumulants
\[
\E_\gamma(\Tr_{\pi_\peq}(\vec{A}^\eta, \vec{B}^\eta))
=
\mathop{\sum_{\cU \in \cP(m+n)}}_{\cU \leq \gamma}
\rk_\cU(\Tr_{\pi_\peq}(\vec{A}^\eta, \vec{B}^\eta)).
\]
Let $u$ be the number of blocks of $\cU$ that contain only
one cycle of $\pi_\peq$. If $\cU$ has a singleton then the
corresponding cumulant will be 0 because the $A$'s and $B$'s
are centred; so we only consider $\cU$'s which have no
singletons and thus $\#(\cU) \leq (m+n)/2$. If we let $u$ be
the number of blocks of $\cU$ that contain exactly one cycle
of $\pi_\peq$, then
$\rk_\cU(\Tr_\peq(\vec{A}^\eta,\vec{B}^\eta)) = \rO(d^u)$
and $u \leq \#(\cU) \leq (m + n)/2$. Recall that $\langle
\Wg(p), q\rangle - \Wg(\gamma, p, q) = \rO(d^{-(m+n) + \#(p
  \vee q)-1})$

Since $\#(p \vee q) \leq (m+n)/2$ we have
\begin{align*}
\big\{\Wg(\gamma, p, q) - &\Wg(\gamma, p, q)\big\}
\rk_\cU(\Tr_{\pi_\peq}(\vec{A}^\eta, \vec{B}^\eta)) \\
& = 
\rO(d^{-(m+n) + \#(p \vee q) -1 + u}) = \rO(d^{-1}).
\end{align*}
Then summing over all $\cU$'s we have
\[
\big\{\Wg(\gamma, p, q) - \Wg(\gamma, p, q)\big\}
\E_\gamma(\Tr_{\pi_\peq}(\vec{A}^\eta, \vec{B}^\eta))
= \rO(d^{-1}).
\]
Thus the expression (\ref{expression:secondterm})
\[
\mathop{\sum_{p, q \in \cP_2(m+n)}}_{p, q \leq \gamma}
\kern-1em
\big\{ \langle \Wg(p), q\rangle - \Wg(\gamma, p, q) \big\}
\E(\Tr_{\pi_\peq}(\vec{A}^\eta,\ab \vec{B}^\eta)). 
\]
converges to $0$.

Let us finally consider the expression (\ref{expression:thirdterm})
\[
\mathop{\sum_{p, q \in \cP_2(m+n)}}_{p, q \leq \gamma}
\kern-1.5em \Wg(\gamma, p, q)
\big\{\E(\Tr_{\pi_\peq}(\vec{A}^\eta,\ab \vec{B}^\eta))
- \E_\gamma(\Tr_{\pi_\peq}(\vec{A}^\eta,\ab \vec{B}^\eta))
\big\}.
\]
For each $p, q \leq \gamma$ we must show that 
\[
\Wg(\gamma, p, q)
\big\{\E(\Tr_{\pi_\peq}(\vec{A}^\eta,\ab \vec{B}^\eta))
- \E_\gamma(\Tr_{\pi_\peq}(\vec{A}^\eta,\ab \vec{B}^\eta))
\big\} = \rO(d^{-1}).
\]
So fix $p, q \leq \gamma$ and write
$\Tr_{\pi_\peq}(\vec{A}^\eta,\ab \vec{B}^\eta) = X_1\cdots
X_r X_{r+1} \cdots X_{r+s}$ with $X_1, \dots, X_r$ coming
from the cycles of $\pi_\peq$ contained in $[m]$ and
$X_{r+1}, \dots, X_{r+s}$ coming from the cycles of
$\pi_\peq$ contained in $[m+1, m+n]$. Then
\[
\E(\Tr_{\pi_\peq}(\vec{A}^\eta,\ab \vec{B}^\eta))
- \E_\gamma(\Tr_{\pi_\peq}(\vec{A}^\eta,\ab \vec{B}^\eta))
= \rk_2(X_1 \cdots X_r, X_{r+1} \cdots X_{r+s}).
\]
Using the formula of Leonov and Shiryaev \cite{ls}
\[
\rk_2(X_1 \cdots X_r, X_{r+1} \cdots X_{r+s})
=
\mathop{\sum_{\cV \in \cP(r+s)}}_{\cV \vee \tau = 1_{r+s}}
\rk_\cV(X_1, \cdots, X_{r+s})
\]
where $\tau = \{(1, \dots, r)(r+1, \dots, r+s)\}$. Now let
us use Notation \ref{notation:cumulantoftrace} to write this
as
\[
\E(\Tr_{\pi_\peq}(\vec{A}^\eta,\ab \vec{B}^\eta))
- \E_\gamma(\Tr_{\pi_\peq}(\vec{A}^\eta,\ab \vec{B}^\eta))
= \kern-1.5em
\mathop{\sum_{\pi_\peq \leq \cU \in \cP(m+n)}}_{\cU \vee \gamma = 1_{m+n} }
\kern-1em
\rk_\cU(\Tr_{\pi_\peq}(\vec{A}^\eta,\ab \vec{B}^\eta)).
\]

If $\cU$ has a singleton $(k)$ then $\pi_\peq$ will have a
singleton $(k)$. As in the proof of Theorem
\ref{theorem:firstorderfreeness} this singleton must be a
centred $A_k$ (or $B_{k-m}$ if $k > m$). So if $\cU$ has a
singleton we must have
$\rk_\cU(\Tr_{\pi_\peq}(\vec{A}^\eta,\ab \vec{B}^\eta)) =
0$. Thus we may assume that $\cU$ has no singletons, so in
particular $\#(\cU) \leq (m+n)/2$. . As before let $u$ be
the number of blocks of $\cU$ that contain exactly one cycle
of $\pi_\peq$. Then
\[
\rk_\cU(\Tr_{\pi_\peq}(\vec{A}^\eta,\ab \vec{B}^\eta)) = \rO(d^u).
\]
Now $u \leq \#(\cU) \leq (m+n)/2$ and, as usual, 
\[\Wg(\gamma, p, q) = \rO(d^{-(m+n) + \#(p \vee q)}).\]
Thus
\[\Wg(\gamma, p, q) \rk_\cU(\Tr_{\pi_\peq}(\vec{A}^\eta,\ab \vec{B}^\eta))
= \rO(d^{-(m+n) + \#(p \vee q) + u}).\] Since $\pi_\peq \leq
\gamma$ and $\cU \vee \gamma = 1_{m+n}$ we must have $u <
(m+n)/2$, as equality would force $\pi_\peq = \cU$ as
partitions. Thus $-(m+n) + \#(p \vee q) + u \leq -1$. Hence
\[\Wg(\gamma, p, q) \rk_\cU(\Tr_{\pi_\peq}(\vec{A}^\eta,\ab \vec{B}^\eta))
= \rO(d^{-1}).\] 
Summing over all $\cU$'s we have
\[
\Wg(\gamma, p, q)
\big\{\E(\Tr_{\pi_\peq}(\vec{A}^\eta,\ab \vec{B}^\eta))
- \E_\gamma(\Tr_{\pi_\peq}(\vec{A}^\eta,\ab \vec{B}^\eta))
\big\} = \rO(d^{-1}).
\]

\end{proof}

\begin{remark}
The proof of Theorem \ref{theorem:secondorderlimit} actually
proves a stronger statement than was claimed. Let $A_1,
\dots, A_s$ is an ensemble of $d \times d$ centred random
matrices where for $\eta \in \{-1, 1\}$ we let $A^{\eta} =
A^t$ for $\eta = -1$ we let $A^\eta = A_j$ for $\eta =
1$. Suppose that for any monomials $W_k =
A_{i_1,k}^{\eta_{i_1,k}} \cdots A_{i_{n_k},k}^{\eta_{n_k,
    k}}$, we have
\begin{itemize}
\item
$\E(\tr(W_i)) = \rO(d^0)$ and

\item 
$\rk_r( \Tr(W_{i_1}), \dots, \Tr(W_{i_r})) = \rO(d^0)$ for $r \geq 2$. 
\end{itemize}
Then by equation (\ref{equation:highestorderterms}) we have
for $m \not = n$
\[
\cov(\Tr(O^{\epsilon_1}A_{i_1} \cdots O^{\epsilon_m}A_{i_m}),
     \Tr(O^{\eta_1}A_{j_1} \cdots O^{\eta_n}A_{j_n})) = \rO(d^{-1})
\] 
and by equations (\ref{equation:intermediatestraightstep})
and (\ref{equation:intermediatereversedstep}) we have for $m
= n$

\begin{multline*}
\cov(\Tr(O^{\epsilon_1}A_{i_1} \cdots O^{\epsilon_m}A_{i_m}),
     \Tr(O^{\eta_1}A_{j_1} \cdots O^{\eta_m}A_{j_m})) \\
= \sum_{s=1}^m  \left\{ \prod_{r=1}^m \E(\tr(A_{i_r} A_{j_{s-r}})) 
                                      \E(\tr(O^{\epsilon_r+\eta_{s-r+1}}))  
                                      \right.\\
 \mbox{} +
\left.\prod_{r=1}^m \E(\tr(A_{i_r} A_{j_{s+r}}^t)) \E(\tr(O^{\epsilon_r-\eta_{s-r}}))\right\} + \rO(d^{-1}),
\end{multline*}
where the indices of the $j$'s and $\eta$'s are interpreted modulo $m$. 
\end{remark}

\begin{corollary}\label{corollary:secondorderlimitorthogonal}
Let $O$ be a $d \times d$ Haar distributed random orthogonal
matrix. Then for integers $m$ and $n$
\[
\lim_{d \rightarrow \infty} 
\cov( \Tr(O^m), \Tr(O^n))
= \begin{cases} 0 & |m|  \not = |n|\\ 2|m| & |m| = |n| \end{cases}.
\]
\end{corollary}

\begin{proof}
Let $\epsilon_1 = \cdots = \epsilon_m = \sgn(m)$ and
$\epsilon_{m+1} = \cdots = \epsilon_{m+n} = \sgn(n)$. Let
$\gamma$ be the permutation with the two cycles $(1,2,
\dots, m)(m+1, \dots, m+n)$.  Then by Proposition
\ref{proposition:exact_formula}
\[
\E(\Tr(O^{\epsilon_1} \cdots O^{\epsilon_m})
\Tr(O^{\epsilon_{m+1}} \cdots O^{\epsilon_{m+n}})) =
\kern-1em \sum_{p,q\in \cP_2(m+n)} \langle \Wg(p), q \rangle
d^{\#(\pi_{\peq})}.
\]
and if let $\cU$ be the partition with blocks the cycles of
$\gamma$
\begin{eqnarray*}\lefteqn{
\E(\Tr(O^{\epsilon_1} \cdots O^{\epsilon_m}))
\E(\Tr(O^{\epsilon_{m+1}} \cdots O^{\epsilon_{m+n}})}\\
&=& \kern-1em
\mathop{\sum_{p,q\in \cP_2(m+n)}}_{p, q \leq \cU} 
\Wg(\cU, p, q) d^{\#(\pi_{\peq})}.
\end{eqnarray*}
By the multiplicativity of the coefficient of the term of
leading order of $\langle \Wg(p), q\rangle$ we thus have
\begin{eqnarray*}\lefteqn{
\cov(\Tr(O^{\epsilon_1} \cdots O^{\epsilon_m}),
     \Tr(O^{\epsilon_{m+1}} \cdots O^{\epsilon_{m+n}}))} \\
& = &
\mathop{\sum_{p,q\in \cP_2(m+n)}}_{p, q \leq \cU}
\langle \Wg(p), q\rangle d^{\#(\pi_\peq)} +
\rO(d^{-(m+n) + \#(p\vee q) + \#(\pi_\peq) - 1})
\end{eqnarray*}
As in the proof of Theorem \ref{theorem:firstorderfreeness},
if $\pi_\peq$ has a singleton $(k)$ then $\epsilon_k =
-\epsilon_{\gamma(k)}$, which is impossible given our
construction of $\epsilon$. Thus $\pi_\peq$ has no
singletons. Hence $\#(\pi_\peq) \leq (m + n)/2$. Thus
$-(m+n) + \#(p \vee q) + \#(\pi_\peq) \leq 0$, with equality
only if $p = q$ and $\pi_\pep$ is a pairing.

Let $(r, s) \in p$. Either $\epsilon_r = -\epsilon_s$ or
$\epsilon_r = \epsilon_s$. As in the proof of Theorem
\ref{theorem:secondorderlimit} all cycles of $p$ connect the
two cycles of $\gamma$ and hence $|m| = |n|$. Also in the
case in which $\epsilon_r = -\epsilon_s$, we have
$(\gamma^{-1}(r), \gamma(s)) \in p$. There are exactly $|m|$
such $p$'s. In the case $\epsilon_r = \epsilon_s$, we have
$(\gamma^{-1}(r), \gamma^{-1}(s)) \in p$. There are exactly
$|m|$ such $p$'s. All together there are $2|m|$ such
$p$'s. By Remark \ref{remark:leadingorder} the coefficient
of $d^{-n/2}$ in $\langle\Wg(p), p \rangle$ is 1. This gives
the claimed result.
\end{proof}


\section{Vanishing of Higher Cumulants of Traces}
\label{section:vanishing}

Let $\{ A_j \}_j$ be a family of $d \times d$ random
matrices, containing the identity matrix, with a real second
order limit distribution. By this we mean that as $d
\rightarrow \infty$

\setbox1=\hbox{%
\begin{minipage}[c]{355pt}
\begin{itemize}

\item $\tr(A_{i_1}^{(\epsilon_1)} \cdots
  A_{i_n}^{(\epsilon_n)})$ converges to
  $\phi(a^{(\epsilon_1)}_{i_1} \cdots
  a_{i_n}^{(\epsilon_n)})$ for all $i_1, \dots, i_n$ and all
  $\epsilon_1, \dots, \epsilon_n$;

\item $\rk_2(\Tr(A_{i_1}^{(\epsilon_1)} \cdots
  A_{i_{m}}^{(\epsilon_{m})}),
  \Tr(A_{i_{m+1}}^{(\epsilon_{m+1})} \cdots
  A_{i_{m+n}}^{(\epsilon_{m+n})}))$ converges to

\noindent $\phi_2(a_{i_1}^{(\epsilon_1)}\ab \cdots
a_{i_m}^{(\epsilon_m)}, a_{i_{m+1}}^{(\epsilon_{m+1})}
\cdots a_{i_{m+n}}^{(\epsilon_{m+n})})$ for all $i_1, \dots,
i_{m+n}$ and all $\epsilon_1, \dots,\ab \epsilon_{m+n}$;

\item $\rk_r(\Tr(A_{i_1}^{(\epsilon_1)} \cdots
  A_{i_{m_1}}^{(\epsilon_{m_1})}),\!\cdots\!, \Tr(A_{i_{m_1
      + \cdots m_{r-1}+1}}^{(\epsilon_{m_1+ \cdots
      +m_{r-1}+1})} \cdots A_{i_{m_1 + \cdots +
      m_r}}^{(\epsilon_{m_1 + \cdots + m_r})}))$

\noindent
converges to $0$ for all $r \geq 3$, all $i_1, \dots, i_{m_1
  + \cdots + m_r}$ and all $\epsilon_1, \dots,\ab
\epsilon_{m_1 + \cdots + m_r}$.

\end{itemize}
\end{minipage}}
\begin{equation}\label{equation:secondorder}
\kern-3.4em \left.\vcenter{\hsize\wd1
\box1}\right\}
\end{equation}
Let $O$ be a Haar distributed $d \times d$ random orthogonal
matrix whose entries are independent from those of $\{
A_j\}_j$. In this section we shall show that whenever $X_1,
\dots, X_r$ be $r$ random variables where each $X_i$ is one
of the following types:

\setbox1=\hbox{%
\begin{minipage}[c]{320pt}
\begin{itemize}
\item%
$X_i = \Tr(A_k)$ for some $k$; or

\item%
$X_i = \Tr(O^{\epsilon_1}A_{j_1} \cdots O^{\epsilon_n} A_{j_n})$ with 
$\epsilon_k \in \{-1, 1\}$ and such that if $A_{j_k} = I$ then 
$\epsilon_{k-1} = \epsilon_k$, where $\epsilon_{n+1} = \epsilon_1$. 
\end{itemize}
\end{minipage}}

\begin{equation}\label{equation:basicform}
\kern-1em\left.\vcenter{\hsize\wd1
\box1}\right\}
\end{equation}
The the third and higher cumulants of the $X$'s will
converge to 0 as $d \rightarrow \infty$. This, combined with
Theorems \ref{theorem:firstorderfreeness} and
\ref{theorem:secondorderlimit} will show that we have
asymptotic real second order freeness of the $\{A_j\}_j$ and
$O$.

For the rest of this section we shall assume that the $\{A_j
\}_j$ satisfy condition (\ref{equation:secondorder}) and our
goal is to prove the theorem below.

\begin{theorem}\label{theorem:vanishingthird} 
Suppose that $X_1, \dots, X_r$ are of the form
(\ref{equation:basicform}) and $r \geq 3$, then
\begin{equation}\label{equation:vanishingthird}
\lim_{d\rightarrow \infty} \rk_r(X_1, \dots, X_r) = 0.
\end{equation}

\end{theorem}

We prove this theorem by proving the following result where
we strengthen the hypothesis in (\ref{equation:basicform})
by assuming that the non-constant $A$'s are centred.

\setbox1=\hbox{%
\begin{minipage}[c]{300pt}
\begin{enumerate}
\item 
$X_i = \Tr(A_k)$ for some $k$ with $\E(\Tr(X_k)) = 0$; or

\item%
$X_i = \Tr(O^{\epsilon_1}A_{j_1} \cdots O^{\epsilon_n}
  A_{j_n})$ with $\epsilon_k \in \{-1, 1\}$ and such that
  either $\E(\Tr(A_{j_k})) = 0$ or$A_{j_k} = I$ and
  $\epsilon_{k-1} = \epsilon_k$, where $\epsilon_{n+1} =
  \epsilon_1$.
\end{enumerate}
\end{minipage}}

\begin{equation}\label{equation:secondform}
\left.\vcenter{\hsize\wd1
\box1}\right\}
\end{equation}

\begin{theorem}\label{theorem:simplifiedvanishingthird} 
Suppose that whenever $X_1, \dots, X_r$ are of form
(\ref{equation:secondform}) and $r \geq 3$ then
\[
\lim_{d \rightarrow \infty} \rk_r(X_1, \dots, X_r) = 0.
\] 
\end{theorem}

\noindent\textit{Proof of Theorem
  \ref{theorem:vanishingthird} using Theorem
  \ref{theorem:simplifiedvanishingthird}}: We begin by
recalling that the cumulant $\rk_r(X_, \dots, X_r)$ will be
0 whenever an $X_i$ is constant and $r \geq 2$. Recall also
that by our assumption of a second order limit distribution
$\E(\tr(A_i))$ is a convergent function of $d$ and thus
bounded. Thus if $\rk_r(X_1, \dots, X_r) \rightarrow 0$ then
so does $\E(\tr(A_j)) \rk_r(X_1, \dots, X_r)$.

Suppose $X_i = \Tr(A_j)$ for some $j$. Let $\centre{A_j} =
A_j - \E(\tr(A_j))I$. Let $c = \E(\tr(A_j))$. Then
$\E(\Tr(\centre{A_j})) = 0$ and $A_j = \centre{A_j} + cI$.
Then $\rk_r(X_1, \dots, X_{i-1}, cd, X_{i+1}, \dots, X_r) =
0$ and so
\begin{eqnarray*}\lefteqn{
\rk_r(X_1, \dots, X_r) } \\
& = & 
\rk_r(X_1, \dots, X_{i-1}, cd, X_{i+1}, \dots, X_r) \\
&&\mbox{}+
\rk_r(X_1, \dots, X_{i-1}, \Tr(\centre{A_j}), X_{i+1}, \dots, X_r) \\
& = &
\rk_r(X_1, \dots, X_{i-1}, \Tr(\centre{A_j}), X_{i+1}, \dots, X_r).
\end{eqnarray*}
So we may suppose that any $X$'s of the form $\Tr(A_j)$ are centred. 

Next suppose that $X_i = \Tr(O^{\eta_1}A_{j_1} \cdots
O^{\eta_s}A_{j_s})$, with each $\eta_i = \pm1$ and whenever
$A_{j_t} = I$ we have $\eta_t = \eta_{t+1}$. For each $i$,
we shall write $X_i = \Tr(O^{\eta_1}A_{j_1} \ab\cdots
O^{\eta_s}A_{j_s})$ as a linear combination of a constant
random variable and terms of the form $\Tr(A_{j_t})$, or
$\Tr(O^{\mu_1}A_{k_1} \ab\cdots O^{\mu_l}A_{k_l})$ where for
each $t$ either $\E(\Tr(A_{k_t})) = 0$ or $A_{k_t} = I$ and
$\mu_t = \mu_{t+1}$; where $\mu_{l+ 1} = \mu_1$. We then
replace $X_i$ in $\rk_r(X_1, \dots, X_r)$ by this linear
combination and get a sum of cumulants in which all the
$A$'s are of the form (\ref{equation:secondform}).

To show that each $X_i = \Tr(O^{\eta_1}A_{j_1} \cdots
O^{\eta_s}A_{j_s})$ can be written as such a linear
combination we replace for each $t$, $A_{j_t}$ with
$\centre{A}_{j_t} + \E(\tr(A_{j_t})) I$. We then expand this
sum. If we have a factor $\E(\tr(A_{j_t}))I$, we will get
cancellation of cyclically adjacent $O$'s wherever $\eta_t =
- \eta_{t+1}$. This might bring two centred $A$'s next to
each other. As the product will not necessarily be such the
expectation of the trace is 0, we repeat the centring
process and continue. Since the number of factors decreases
whenever there is a cancellation, the process terminates
with either: an $X_i$ of the form
(\ref{equation:secondform},\textit{i}); an $X_i$ as in
(\ref{equation:secondform}.\textit{ii}); or a constant $X_i$
(if all the $O$'s get cancelled). \qed

\begin{remark}
To illustrate the previous theorem let us consider the
example
\[
\rk_3(\Tr(OA_1O^{-1}A_2), \Tr(OA_3OA_4), \Tr(OA_5O^{-1}A_6)).
\]
There are six $A$'s and we let $A_i = \centre{A}_i + c_iI$
with $c_i = \E(\Tr(A_i))$. This produces $2^6$ terms, some
of which are 0 because some of the entries of the cumulant
are constant. For example we shall get terms such as
\[
c_1 c_3 c_4 c_5 \rk_3(\Tr(\centre{A_2}), \Tr(OIOI), \Tr(\centre{A}_6)). 
\]
If we started with the example
\[
\rk_3(\Tr(OA_1O^{-1}A_2), \Tr(OA_3O^{-1}A_4), \Tr(OA_5O^{-1}A_6)).
\]
then we would also get terms like
\[
c_1c_3c_5\rk_3(\Tr(\centre{A}_2), \Tr(\centre{A}_4), \Tr(\centre{A}_6))
\]
where there no $O$'s.
\end{remark}

Our task now is to prove Theorem
\ref{theorem:simplifiedvanishingthird}. We shall recall the
moment cumulant relation
\begin{equation}\label{equation:momentcumulant}
\E(X_1 \cdots X_r) = \sum_{\cU \in \cP(r)} \rk_\cU(X_1, \dots, X_r).
\end{equation}
So to prove something about the cumulants $\rk_r(X_1, \dots,
X_r)$ we shall prove something first about $\E(X_1 \cdots
X_r)$ and use this to prove
Theorem~\ref{theorem:simplifiedvanishingthird}. We let
$\cP_{1,2}(n)$ be the partitions of $[n]$ with blocks of
size either 1 or 2.

\begin{theorem}\label{theorem:part12version}
Whenever $X_1, \dots, X_r$ are of form
(\ref{equation:secondform}) then
\begin{equation}\label{equation:part12version}
\E(X_1 \cdots X_r) = \sum_{\cU \in \cP_{1,2}(r)} \rk_\cU(X_1, \dots, X_r)
+ \lo(1).
\end{equation}  

\end{theorem}

\noindent\textit{Proof of Theorem
  \ref{theorem:simplifiedvanishingthird} using Theorem
  \ref{theorem:part12version}.}

\medskip\noindent By Corollary
\ref{corollary:firstcumulants} we have that $\rk_1(X_i) =
\rO(1)$ is $X_i$ is of type
(\ref{equation:secondform}.\textit{ii}) and $\rk_1(X_i) = 0$
if $X_i$ is of type
(\ref{equation:secondform}.\textit{i}). If $X_{i_1}$ and
$X_{i_2}$ are both of type
(\ref{equation:secondform}.\textit{ii}) then by Theorem
\ref{theorem:secondorderlimit}, $\rk_2(X_{i_1}, X_{i_2}) =
\rO(1)$. If they are both of type
(\ref{equation:secondform}.\textit{i}), then by assumption
(\ref{equation:secondorder}) we have $\rk_2(X_{i_1},
X_{i_2}) = \rO(1)$. If $X_{i_1}$ is of type
(\ref{equation:secondform}.\textit{i}) and $X_{i_2}$ is of
type (\ref{equation:secondform}.\textit{ii}), then
$\rk_2(X_{i_1}, X_{i_2}) = \E(X_{i_1}X_{i_2})$, as
$\E(X_{i_1}) = 0$. Then by Corollary
\ref{corollary:firstcumulants}, $\E(X_{i_1}X_{i_2}) =
\rO(d^{-1})$. So in all cases $\rk_1(X_{i_1})$ and
$\rk_2(X_{i_1}, X_{i_2})$ are of order at most $\rO(1)$.

Now by (\ref{equation:part12version})
\begin{eqnarray*}\lefteqn{
\rk_3(X_{i_1}, X_{i_2}, X_{i_3}) }\\ &=&
  \E(X_{i_1}X_{i_2}X_{i_3}) - \sum_{\cU \in \cP_{1,2}(3)}
  \rk_\cU(X_{i_1}, X_{i_2}, X_{i_3}) = \lo(1).
\end{eqnarray*}
Suppose we have shown for $3 \leq s < l$ that
$\rk_s(X_{i_1}, \dots,X_{i_s}) = \lo(1)$. Then
\begin{eqnarray*}\lefteqn{
\E(X_{i_1}\cdots X_{i_l}) - \sum_{\cU \in \cP_{1,2}(l)}
\rk_\cU(X_{i_1}, \dots, X_{i_l})} \\ & = & \rk_l(X_{i_1},
  \dots, X_{i_l}) + \sum_{\cU \in \widetilde{\cP}_{1,2}(l)}
  \rk_\cU(X_{i_1}, \dots, X_{i_l})
\end{eqnarray*}
Where $\widetilde{\cP}_{1,2}(l)$ is all the partitions in
$\cP(l)$ except those in $\cP_{1,2}(l)$ and $1_l$, the
partition with only one block. If $\cU \in
\widetilde{\cP}_{1,2}(l)$ then $\cU$ has blocks of size 1 or
2 and at least one block of size between 3 and $s$. Since
the cumulants from the blocks of order $\rO(1)$ and, by our
induction hypothesis, all others are of order $\rO(d^{-1})$,
the product $\rk_\cU(X_{i_1}, \dots, X_{i_l})$ is of order
$\lo(1)$. Hence
\[
\rk_l(X_{i_1}, \dots, X_{i_l}) + \sum_{\cU \in
  \widetilde{\cP}_{1,2}(l)} \rk_\cU(X_{i_1}, \dots, X_{i_l})
= \lo(1)
\]
forces us to conclude that $\rk_l(X_{i_1}, \dots, X_{i_l}) =
\lo(1)$.  \qed

\begin{notation}\label{notation:rearrangedsetup}
From now on we shall assume that we have positive integers
$n_1, \dots, n_r$. We let $n = n_1 + \cdots + n_r$. There is
$1 \leq r_0 \leq r$ such that for $r_0 \leq i \leq r$ we
have $n_i = 1$. We let $\gamma \in S_n$ be the permutation
with cycles
\[
(1, \dots, n_1) \cdots (n_1 + \cdots + n_{r_0 -1} +1, \dots,
n_1 + \cdots + n_{r_0})
\]
\[
\mbox{} \times
(n_1 + \cdots + n_{r_0} + n_{r_0 + 1}) \cdots
(n_1 + \cdots + n_{r_0} + n_r)
\]
If $r_0 = 1$ then $\gamma = e$ is the identity
permutation. We shall assume the random variables $X_i$ are
such that for $1 \leq i \leq r_0$
\begin{itemize}
\item $X_i = \Tr(%
      O^{\epsilon_{n_1 + \cdots + n_{i-1}+1}}
      A_{n_1 + \cdots + n_{i-1}+1} \cdots O^{\epsilon_{n_1 +
      \cdots + n_i}} A_{n_1 + \cdots + n_i})$
  \goodbreak{}\noindent where for each $n_1 + \cdots +
  n_{i-1}+1 \leq t \leq n_1 + \cdots + n_i$ either
  $\E(\Tr(A_t)) = 0$ or $A_t = I$ and $\epsilon_t =
  \epsilon_{\gamma(t)}$;
\end{itemize}
and for $r_0 < i \leq r$
\begin{itemize}
\item
$X_i = \Tr(A_{n_1 + \cdots + n_i})$ and $\E(X_i) = 0$.
\end{itemize}
Let $m = n_1 + \cdots + n_{r_0-1}$. If $m$ is odd and
positive then $\E(X_1 \cdots X_r) = 0$. So we shall assume
that $m$ is even, and possibly 0. Let $\cP_2(m, n)$ be
the set of partitions of $[n]$ whose restriction to $[m]$ is
a pairing and all of whose other blocks are singletons. In
the case $r_0 = 1$ we have $m = 0$ and the only partition in
$\cP_2(m,n)$ is the one with $n$ blocks of size 1. We
assume that $\epsilon \in \bZ_2^n$ with $\epsilon_i = 1$ for
$i > m$.

Now let $p$ and $q$ be in $\cP_2(m,n)$. Then $p \delta q
\delta$ is a permutation of $[\pm n]$ whose restriction to
$[\pm m]$ is a pairing and all of whose other cycles are
singletons. Now consider $\gamma_-^{-1} \delta_\epsilon p
\delta q \delta \delta_\epsilon \gamma$. Its restriction to
$[\pm n] \setminus [\pm m]$ consists of singletons. Its
restriction to $[\pm m]$ is as in Notation
\ref{notation:krewerascomplement}, i.e. the cycles occur in
pairs $\{c, c'\}$. We obtained a permutation, $\pi_\peq$, of
$[m]$ as follows. For each pair we choose one
representative, replacing any negative entries by their
absolute values. Now we wish to extend this construction to
the case where $p, q \in \cP_2(m,n)$. The cycles in
$[\pm n] \setminus [\pm m]$ also occur in pairs $(-k)(k)$
(with $k > 0$) and so we just choose $(k)$ for each of these
cycles. Also for $m < k \leq n$ let $\eta_\peq(k) = 1$.
\end{notation}

Let $X_1, \dots, X_r$ satisfy (\ref{equation:secondform})
and let us expand $\E(X_1 \cdots X_r)$ as follows.
\begin{eqnarray*}\lefteqn{
\E(X_1 \cdots X_r)}\\
& = &
\sum_{p,q \in \cP_2(m,n)}
\langle \Wg(p), q \rangle
\E(\Tr_{(\pi_\peq, \eta_\peq)}(A_1, \dots, A_n)).
\end{eqnarray*}
We need to find the order of $\E(\Tr_{(\pi_\peq,
  \eta_\peq)}(A_1, \dots, A_n))$.
\begin{proposition}\label{proposition:stepIV}
If $m \geq 2$ and ${\pi_\peq}|_{{[m]}}$ is not a pairing then 
\[
\E(\Tr_{(\pi_\peq, \eta_\peq)}(A_1, \dots, A_n)) = \rO(d^{m/2 - 1}).
\]
If $\pi_\peq|_{[m]}$ is a pairing or if $m = 0$ then 
\begin{eqnarray*}\lefteqn{
\E(\Tr_{(\pi_\peq, \eta_\peq)}(A_1, \dots, A_m)) }\\ & = &
  \E_{\pi_\peq|_{[m]}}(\Tr_{(\pi_\peq|_{[m]},
    \eta_\peq|_{[m]})}(A_1, \dots, A_m))\E(\Tr(A_{m+1})
  \cdots \Tr(A_n)) \\ && \quad \mbox{} + \rO(d^{m/2-1}).
\end{eqnarray*}
\end{proposition}

\begin{proof}
Let follow the notation used in Equation
(\ref{equation:partitionnotation}). If $\cU$ is a partition
on $[n]$ and $\pi$ any permutation of $[n]$ we write
$\E_\cU(\Tr_\pi(A_1, \dots, A_n))$ to be the product
$\prod_{i=1}^k \E(\Tr_{\pi_i}(A_1, \dots, A_n))$, where the
blocks of $\cU$ are $\{U_1, \dots, U_k\}$ and $\pi_i =
\pi|_{U_i}$. We likewise let $\rk_\cU(\Tr_\pi(A_1, \dots,
A_n))$ be the product of cumulants along the blocks of
$\cU$, see equation
(\ref{equation:cumulantnotation}). Recall that we then have
the moment-cumulant relation
\begin{eqnarray}\label{equation:moebiusinversionii}\lefteqn{
\E(\Tr_{(\pi_\peq, \eta_\peq)}(A_1, \dots, A_n))}\notag\\
& = &
\mathop{\sum_{\cU \in \cP(n)}}_{\pi_\peq \leq \cU}
\rk_\cU(\Tr_{(\pi_\peq, \eta_\peq)}(A_1, \dots, A_n)).
\end{eqnarray}
By our assumption (\ref{equation:secondorder}) on the
existence of a real second order limit distribution of the
$A$'s we have
\[
\rk_\cU(\Tr_{(\pi_\peq, \eta_\peq)}(A_, \dots, A_n))
=
\rO(d^u)
\]
where $u$ is the number of blocks of $\cU$ that contain only
one cycle of $\pi_\peq$. Suppose $\rk_\cU(\Tr_{(\pi_\peq,
  \eta_\peq)}(A_, \dots, A_n)) \not= 0$. Then any block of
$\cU$ that contains a single cycle of $\pi_\peq$ must
contain a cycle of $\pi_\peq|_{[n]}$, as $\E(\Tr(A_k)) = 0$
for $m < k \leq n$. Thus $u \leq \#(\pi_\peq|_{[m]})$. Also
recall, from the fourth paragraph of the proof of Theorem
\ref{theorem:firstorderfreeness}, that if $(k)$ is a
singleton of $\pi_\peq|_{[m]}$ then $\E(\Tr(A_k)) = 0$, and
hence $\rk_\cU(\Tr_{(\pi_\peq, \eta_\peq)}(A_1, \dots,\ab
A_n)) = 0$.  So for any block $U$ of $\cU$ that contains
only one cycle of $\pi_\peq$, $U$ must contain at least two
elements. Thus $u \leq m/2$. We can only have $u = m/2$ when
every block of $\cU|_{[m]}$ contains one cycle of
$\pi_\peq|_{[m]}$ and that cycle has two elements,
i.e. $\pi_\peq|_{[m]}$ is a pairing. This proves the first
claim.

If $\pi_\peq|_{[m]}$ is a pairing then we have just seen
that to have $u = m/2$ we must have $\cU|_{[m]} = \pi_\peq
|_{[m]}$. Thus if we only consider $\cU$'s for which
$\cU|_{[m]} = \pi_\peq |_{[m]}$ we have
\begin{eqnarray*}\lefteqn{
\E(\Tr_{(\pi_\peq, \eta_\peq)} (A_1, \dots, A_n))}\\ & = &
  \E_{\pi_\peq|_{[m]}}(\Tr_{(\pi_\peq|_{[m]},
    \eta_\peq|_{[m]})}(A_1, \dots, A_m)) \\ && \quad\mbox{}
  \times \sum_{\cU' \in \cP([m+1, n])} \kern-1.5em
  \rk_{\cU'}(\Tr(A_{m+1}), \dots, \Tr(A_m)) \\ & = &\!\!\!
  \E_{\pi_\peq|_{[m]}}(\Tr_{(\pi_\peq|_{[m]},
    \eta_\peq|_{[m]})}(A_1, \dots, A_m)) \E(\Tr(A_{m+1}),
  \dots, \Tr(A_m))
\end{eqnarray*}
Finally we add back the remaining terms to obtain that
\begin{eqnarray*}\lefteqn{
\E(\Tr_{(\pi_\peq, \eta_\peq)} (A_1, \dots, A_n))}\\ & =
  &\!\!\!  \E_{\pi_\peq|_{[m]}}(\Tr_{(\pi_\peq|_{[m]},
    \eta_\peq|_{[m]})}(A_1, \dots, A_m)) \E(\Tr(A_{m+1}),
  \dots, \Tr(A_m)) \\ && \quad\mbox{} + \rO(d^{m/2 - 1}).
\end{eqnarray*}
\end{proof}

\begin{notation} 
Suppose we have $r_0, r, m, n$ and $\gamma$ and $\epsilon$
as in Notation \ref{notation:rearrangedsetup}. Let
$\cA(\gamma,\epsilon,m,n)$ be the set of partitions $p \in
\cP_2(m,n)$ such that $\pi_\pep|_{[m]}$ is a pairing,
the condition being vacuously satisfied when $m = 0$. For $p
\in \cA(\gamma, \epsilon, m, n)$ let
\begin{eqnarray*}\lefteqn{
\cE_p(A_1, \dots, A_n)}\\ & = & d^{-m/2}
  \E_{\pi_\pep|_{[m]}}(\Tr_{(\pi_\pep|_{[m]},
    \eta_\pep|_{[m]})}(A_1, \dots, A_m)) \\ &&\quad\mbox{}
  \times \E(\Tr(A_{m+1}) \cdots \Tr(A_n)).
\end{eqnarray*}
\end{notation}

\begin{corollary}\label{corollary:scriptE}
Suppose $X_1, \dots, X_r$ satisfy
(\ref{equation:secondform}). Then
\[
\E(X_1 \cdots X_r) =
\sum_{p \in \cA(\gamma, \epsilon,m ,n)}
\cE_p(A_1, \dots, A_n) + \rO(d^{-1}).
\]
\end{corollary}

\begin{proof} When $m = 0$ there is nothing to prove. 
According to Proposition
\ref{proposition:exact_formulafirstextension}
\begin{eqnarray*}\lefteqn{
\E(X_1 \cdots X_r) }\\
& = &
\sum_{p, q \in \cP_2(m,n)}
\langle \Wg(p), q \rangle
\E(\Tr_{(\pi_\peq, \eta_\peq)}(A_1, \dots, A_n)). 
\end{eqnarray*}
By Proposition \ref{proposition:stepIV}, if
$\pi_\peq|_{[m]}$ is not a pairing we have
\[
\E(\Tr_{(\pi_\peq, \eta_\peq)}(A_1, \dots, A_n)) = \rO(d^{m/2 - 1})
\]
and $\langle \Wg(p), q \rangle = \rO(d^{-n + \#(p \vee
  q)})$. So
\[
\langle \Wg(p), q \rangle \E(\Tr_{(\pi_\peq,
  \eta_\peq)}(A_1, \dots, A_n)) = \rO(d^{- 1}).
\]
Also if $\#(p \vee q) < n/2$ (i.e. $p \not= q$) we get the
same conclusion. When $p=q$ and $\pi_\pep|_{[m]}$ is a
pairing, then $p \in \cA(\gamma, \epsilon,m,n)$ and
\[
\langle \Wg(p), q \rangle \E(\Tr_{(\pi_\peq,
  \eta_\peq)}(A_1, \dots, A_n)) = \cE_p(A_1, \dots, A_n) +
\rO(d^{- 1})
\]
because $\langle \Wg(p), p\rangle = d^{-m/2} + \rO(d^{-m/2 -
  1})$.
\end{proof}

\medskip\noindent\textit{Proof of Theorem
  \ref{theorem:part12version}:} 
To prove the theorem we show that
\begin{eqnarray}\label{equation:aform}
\sum_{\cU \in \cP_2(r)}
\rk_\cU(X_1, \dots, X_r) 
=
\sum_{p \in\cA(\gamma, \epsilon, m, n)}
\cE_p(A_1, \dots, A_n) + \lo(1)
\end{eqnarray}
and then apply Corollary \ref{corollary:scriptE}. We saw in
the proof of Theorem \ref{theorem:simplifiedvanishingthird}
that if $X_{i_1}$ is of type
(\ref{equation:secondform}.\textit{i}) and $X_{i_2}$ is of
type (\ref{equation:secondform}.\textit{ii}) then
$\rk_2(X_{i_1}, X_{i_2}) = \rO(d^{-1})$, so on the left hand
side of (\ref{equation:aform}) we only have to consider
$\cU$'s for which each block is either contained in $[m]$ or
in $[m+1, n]$. Thus
\begin{eqnarray*}\lefteqn{
\sum_{\cU \in \cP_2(r)} \rk_\cU(X_1, \dots, X_r)} \\
& = &
\sum_{\cU \in \cP_2(r_0)} \rk_{\cU}(X_1, \dots, X_{r_0})
\sum_{\cV \in \cP_2([r_0+1,r])} \rk_{\cV}(X_{r_0+1}, \dots, X_r) 
+ \lo(1).
\end{eqnarray*}

By assumption (\ref{equation:secondorder}) we have 
\[
\sum_{\cV \in \cP_{1,2}([r_0+1, r])} \rk_\cV (X_{r_0+1}, \dots, X_r)
= \E(X_{m+1} \cdots X_n) + \lo(1)
\]
because cumulants corresponding to blocks of size three or
larger are $\lo(1)$ and cumulants corresponding to blocks of
size two are $\rO(1)$ and cumulants corresponding to blocks
of size one are 0.

Let us next show that
\begin{eqnarray}\label{equation:cumulantsandE}\lefteqn{
\sum_{\cU \in \cP_{1,2}(r_0)} \rk_\cU(X_1, \dots, X_{r_0}) }\notag\\
& = &
d^{-m/2}\kern-1.5em
\mathop{\sum_{p \in \cP_2(m)}}_{\pi_\pep \mathrm{\ a\ pairing}}\kern-1em
\E_{\pi_\pep}(\Tr_{(\pi_\pep, \eta_\pep)}(A_1, \dots, A_m)) + \rO(d^{-1}).
\end{eqnarray}
If we multiply these last two equations we get equation
(\ref{equation:aform}) as
\begin{eqnarray*}\lefteqn{
\cE_p(A_1, \dots, A_n) }\\ 
& = &
d^{-m/2}\E_{\pi_\pep}(\Tr_{(\pi_\pep, \eta_\pep)}(A_1, \dots, A_m))
\E(X_{m+1} \cdots X_n).
\end{eqnarray*}
To prove (\ref{equation:cumulantsandE}) we use
(\ref{equation:firstcumulantandp}) and
(\ref{equation:expectationexpansion}). They say that a first
and second cumulant of $X$'s if type
(\ref{equation:secondform}.\textit{ii}) can be written, up
to terms of order $\rO(d^{-1})$, as sums over pairings $p$
in unions of intervals of $\gamma$ for which $\pi_\pep$ is a
pairing. Moreover by Corollary
\ref{corollary:maximalconectivity} if $p \in \cP_2(m)$ is a
pairing and $\pi_\pep$ is also a pairing then at most two
cycles of $\gamma$ can be contained in any block of $p \vee
\gamma$.

Let $p \in \cP_2(m)$ be a pairing such that $\pi_\pep$ is a
pairing. The partition $p \vee \gamma$ determines a
partition $\cU_p \in \cP(r_0)$ of the cycles of $\gamma$.
By Corollary \ref{corollary:maximalconectivity}, $\cU_p \in
\cP_{1,2}(r_0)$. Thus we can write
\[
d^{-m/2}\kern-1.0em
\sum_{p \in \cA(\gamma, \epsilon,m)}
\E_{\pi_\pep}(\Tr_{(\pi_\pep, \eta_\pep)}(A_1, \dots, A_m)) \]
\[
=
\sum_{\cU \in \cP_{1,2}(r_0)}
d^{-m/2}\kern-1.0em
\mathop{\sum_{p \in \cA(\gamma, \epsilon,m)}}_{\cU_p = \cU}
\E_{\pi_\pep}(\Tr_{(\pi_\pep, \eta_\pep)}(A_1, \dots, A_m)).
\]
So to prove (\ref{equation:cumulantsandE}) it suffices to
prove that for $\cU \in \cP_{1,2}(r_0)$
\begin{eqnarray}\label{equation:individualpartition}\lefteqn{
\rk_\cU(X_1, \dots, X_{r_0}) }\notag\\ 
& = &
\mathop{\sum_{p \in \cA(\gamma, \epsilon,m)}}_{\cU_p = \cU}
\E_{\pi_\pep}(\tr_{(\pi_\pep, \eta_\pep)}(A_1, \dots, A_m)) + \rO(d^{-1}).
\end{eqnarray}
Now $\rk_\cU(X_1, \dots, X_{r_0})$ is a product of first and
second cumulants. For each first cumulant, $\E(X_j)$, we
apply equation (\ref{equation:firstcumulantandp}) to write
\[
\E(X_j) = d^{-s/2} \sum_p \E_{\pi_\pep}(\Tr_{(\pi_\pep,
  \eta_\pep)}(A_{i_1}, \dots, A_{i_s})) + \rO(d^{-1})
\] 
with $p$ running over pairings of the corresponding cycle
$(i_1, \dots, i_s)$ of $\gamma$ such that $\pi_\pep$ is a
pairing.

For each second cumulant $\rk_2(X_k, X_l)$ we apply equation
(\ref{equation:expectationexpansion}) to write
\[
\cov(X_k, X_l) = d^{-t/2} \sum_p
\E_{\pi_\pep}(\Tr_{(\pi_\pep, \eta_\pep)}(A_{j_1}, \dots,
A_{j_t})) + \rO(d^{-1})
\] 
with $p$ running over pairings that connect the corresponding union
$(j_1, \dots, j_t)$ of two cycles of $\gamma$ such that
$\pi_\pep$ is a pairing. Taking the product of these
equations gives us (\ref{equation:individualpartition}). \qed


\section{Main Results on Asymptotic\\ 
          Real Second Order Freeness}
\label{section:mainresults}

In this section we will present some consequences of
Theorems \ref{theorem:firstorderfreeness},
\ref{theorem:secondorderlimit} and
\ref{theorem:vanishingthird}.

\begin{theorem}\label{theorem:orthogonalsecondlimit}
The ensemble of Haar distributed orthogonal random matrices
has a real second order limit distribution.
\end{theorem}

\begin{proof}
Corollaries \ref{corollary:haarorthogonaldistribution} and
\ref{corollary:secondorderlimitorthogonal} show that an
ensemble of Haar orthogonal matrices has convergent moments
$\{\E(\tr(O^m))\}_d$ and convergent fluctuation moments
$\{\rk_2(\Tr(O^m), \Tr(O^n))\}_d$. A particular example of
Theorem \ref{theorem:vanishingthird} is the case when we
have $\rk_r(\Tr(O^{m_1}), \dots, \Tr(O^{m_r}))$ for some
non-zero integers $m_1, \dots, m_r$. Together these results
then show that an ensemble of Haar orthogonal random
matrices has a real second order limit distribution.
\end{proof}

\begin{theorem}\label{theorem:a_and_o}
Suppose $\{A_i\}_i$ is an ensemble of random matrices with a
real second order limit distribution and $O$ is an ensemble
of Haar distributed orthogonal random matrices. If the
entries of $\{A_i\}_i$ are independent from those of $O$,
then $\{A_i\}_i$ and $O$ are asymptotically real second
order free.
\end{theorem}

\begin{proof}
This is a consequence of Theorems
\ref{theorem:firstorderfreeness},
\ref{theorem:secondorderlimit} and
\ref{theorem:vanishingthird}.
\end{proof}

\begin{theorem}
Let $O_1, \dots, O_s$ be independent Haar distributed
orthogonal random matrices. Then $O_1, \dots, O_s$ are
asymptotically real second order free.
\end{theorem}

\begin{proof}
A single $O$ has a real second order limit distribution by
Theorem \ref{theorem:orthogonalsecondlimit}. By Theorem
\ref{theorem:a_and_o}, $O_1$ and $O_2$ are asymptotically
real second order free. Again by Theorem
\ref{theorem:a_and_o} $\{O_1, O_2\}$ and $O_3$ are
asymptotically real second order free. By Proposition
\ref{proposition:associativity}, $O_1$, $O_2$, and $O_3$ are
asymptotically real second order free. Then we can proceed
by induction.
\end{proof}

\begin{proposition}\label{prop:c02}
Suppose $ \{ A_i \}_i $ and $ \{ B_j \}_j $ are two
independent families of $ d \times d $ random matrices, each
having a real second order limit distribution, and suppose
that $ O $ is a $ d \times d $ Haar orthogonal matrix
independent from $ \{ A_i \}_i \cup \{ B_j \}_j $.  Then $
\{ B_j \}_l $ and $ \{ O A_i O^{-1} \}_i $ are
asymptotically real second order free.
\end{proposition}

\begin{proof}
We do not know that $\{A_i\}_i \cup \{B_j\}_j$ has a real
second order limit distribution so we cannot directly apply
Theorems \ref{theorem:firstorderfreeness},
\ref{theorem:secondorderlimit} and
\ref{theorem:vanishingthird}. We shall argue that because of
the special nature of the words we are considering, i.e. $O
A_{i_1} O^{-1} B_{j_1} O A_{i_2} O^{-1} B_{j_2} \cdots O
A_{i_n} O^{-1} B_{j_n}$, the proofs can be modified so that
we only need the independence of $\{A_i\}_i$ and $\{B_j\}_j$
and the fact that the exponents of the $O$'s alternate in
sign.

Consider the expression
\[
\E(\Tr_{(\pi_\peq,\eta_\peq)}(Y_1, \dots, Y_n)) 
\] 
appearing in the statement of Proposition
\ref{proposition:exact_formula}. If we write
\[
\Tr_{(\pi_\peq,\eta_\peq)}(Y_1, \dots, Y_n) = 
\Tr(Z_1) \cdots \Tr(Z_k)
\] 
as a product along the cycles $c_1 \cdots c_k$ os
$\pi_\peq$, then the existence of a real second order limit
distribution was used to conclude that
\begin{align*}
d^{-1} \rk_1(\Tr(Z_{i})) & \mbox{\ converges},\\
\rk_2(\Tr(Z_{i}), \Tr(Z_{j})) & \mbox{\ converges}, \mbox{\ and} \\
\rk_r(\Tr(Z_{i_1}), \dots, \Tr(Z_{i_r})) &= \lo(1) 
\mbox{\ for\ } r \geq 3.
\end{align*}
This was all we needed to prove Theorems
\ref{theorem:firstorderfreeness},
\ref{theorem:secondorderlimit} and
\ref{theorem:vanishingthird}. We shall show that we still
have these three properties even though we do not assume
that $\{A_i\}_i \cup \{B_j\}_j$ has a real second order
limit distribution.

So let $n_1, n_2, \dots, n_r$ be even positive integers and
$n = n_1 + \cdots + n_r$. Let
\[
\gamma =
(1, \dots, n_1)(n_1+1, \dots, n_1 + n_2) \cdots
(n_1 + \cdots + n_{r-1}+1, \dots, n_1 + \cdots  n_r)
\]
be the permutation in $S_n$ with the cycle decomposition
given above. Let $Y_1, Y_3, \dots, Y_{n-1}$ be polynomials
in $\{A_i\}_i$ and $Y_2, Y_4, \dots , Y_n$ be polynomials in
$\{B_j\}_j$. By Proposition \ref{proposition:exact_formula}
we have
\[
\E(\Tr_\gamma(OY_1,O^{-1}Y_2, \dots, OY_{n-1},O^{-1}Y_n))
\]
\[
=
\sum_{p,q \in \cP_2(n)} \langle \Wg(p), q \rangle
\E(\Tr_{(\pi_\peq, \eta_\peq)}
(Y_1, \dots, Y_n).
\]

Now by Lemma \ref{lemma:evenodd} 
\[
\E(\Tr_{(\pi_\peq, \eta_\peq)}
(Y_1, \dots, Y_n)
=
\E(\Tr(Z_1) \cdots \Tr(Z_k))
\]
where each $Z_i$ is a polynomial in either $\{A_i\}_i$ or in
$\{B_j\}_j$. In fact we may suppose that $Z_1, \dots, Z_l$
are polynomials in $\{A_i\}_i$ and $Z_{l+1}, \dots, Z_k$ are
polynomials in $\{B_j\}_j$. Then we have
\begin{multline}\label{equation:alternatingexpansion}
\E(\Tr_\gamma(OY_1,O^{-1}Y_2, \dots, OY_{n-1},O^{-1}Y_n)) \\
=
\sum_{p,q \in \cP_2(n)} \langle \Wg(p), q \rangle
\E(\Tr(Z_1) \cdots \Tr(Z_l)) \\
\mbox{} \times \E(\Tr(Z_{l+1}) \cdots \Tr(Z_k))
\end{multline}
by the independence of the $\{A_i\}_i$ and the
$\{B_j\}_j$. This means that as far as the asymptotic
behaviour of $\E(\Tr_{(\pi_\peq, \eta_\peq)} (Y_1, \dots,
Y_n))$ is concerned we may assume that $\{A_i\}_i \cup
\{B_j\}_j$ does have a real second order limit
distribution. Now having cleared this hurdle we have by the
proof of Theorem \ref{theorem:firstorderfreeness} that
$\{OA_iO^{-1}\}_i$ and $\{B_j\}_j$ are first order
free. Likewise, the proof of Theorem
\ref{theorem:vanishingthird} can be applied to conclude that
all third and higher cumulants of traces of products of
$OA_iO^{-1}$'s and $B_j$'s are of order $\lo(1)$ as $d
\rightarrow \infty$. We shall conclude the proof by showing
that Theorem \ref{theorem:secondorderlimit} and equation
(\ref{equation:alternatingexpansion}) will give us condition
(\textit{ii}) of
Definition~\ref{definition:secondorderfreeness}.

So let us consider centred random matrices $X_1, \dots,
X_{m}$ and $Y_1, \dots Y_n$ where $X_1,\ab X_3,\ab \dots
X_{m-1}$ and $Y_1, Y_3, \dots Y_{n-1}$ are polynomials in
$\{A_i\}_i$ and $X_2, X_2, \dots, X_m$ and $Y_2, Y_4, \dots,
Y_{n}$ are polynomials in $\{B_j\}_j$. Let the second order
limit distribution of $X_1, \dots, X_m$ and $Y_1, \dots,
Y_{n}$ be given by $x_1, \dots, x_m$ and $y_1, \dots, y_{n}$
respectively.

By equation  (\ref{equation:secondorderlimit}) we have for $m=n$
\[
\lim_{d \rightarrow \infty}
\cov(\Tr(OX_1O^{-1} X_2 \cdots O^{-1}X_m), \Tr(OY_{1} \cdots O^{-1}Y_{n}))
\]
\[=
\sum_{r=1}^{m}\left\{ \prod_{i=1}^m \phi(x_i y_{r-i})
\delta_{\epsilon_i, -\epsilon_{\gamma^{-i+1}(m+r)}}
+ \prod_{i=1}^m \phi(x_i y_{r+i}^t)
\delta_{\epsilon_i, \epsilon_{\gamma^{i}(m+r)}}\right\}
\]
Since $\epsilon_i = (-1)^i$, we have both
$\delta_{\epsilon_i, -\epsilon_{\gamma^{-i+1}(m+r)}} = 1$
and $\delta_{\epsilon_i, \epsilon_{\gamma^{i}(m+r)}} = 1$
only when $r$ is even; see Figure
\ref{figure:paritypreserved}. Thus
\begin{multline}\label{equation:paritypreserved}
\sum_{r=1}^{m}\left\{ \prod_{i=1}^m \phi(x_i y_{r-i})
\delta_{\epsilon_i, -\epsilon_{\gamma^{-i+1}(m+r)}}
+ \prod_{i=1}^m \phi(x_i y_{r+i}^t)
\delta_{\epsilon_i, \epsilon_{\gamma^{i}(m+r)}}\right\} \\
=
\sum_{r=1}^{m/2}\left\{ \prod_{i=1}^m \phi(x_i y_{2r-i})
+ \prod_{i=1}^m \phi(x_i y_{2r+i}^t)\right\}
\end{multline}
For $i$ odd we write $\phi(x_iy_{2r-i}) = \phi(ox_io^{-1}
oy_{2r-i}o^{-1})$. Then
\begin{multline*}
\sum_{r=1}^{m/2}\left\{ \prod_{i=1}^m \phi(x_i y_{2r-i})
+ \prod_{i=1}^m \phi(x_i y_{2r+i}^t)\right\} \\
=
\sum_{r=1}^{m/2}\left\{ \prod_{i=1}^{m/2} 
\phi\big((o x_{2i-1}o^{-1})(o y_{2r-(2i-1)}o^{-1})\big)
\phi(x_{2i} y_{2r-2i}) \right.\\
\left.\mbox{}+ 
\prod_{i=1}^{m/2} \phi\big((ox_{2i-1} o^{-1})(oy_{2r+2i-1}o)^t\big)
\phi(x_{2i} y_{2r + 2i}^t)\right\}
\end{multline*}
This shows that condition (\textit{ii}) of
Definition~\ref{definition:secondorderfreeness} is
satisfied.
\begin{figure}
\hfill\includegraphics{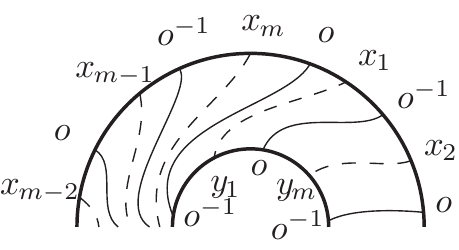}\hfill\includegraphics{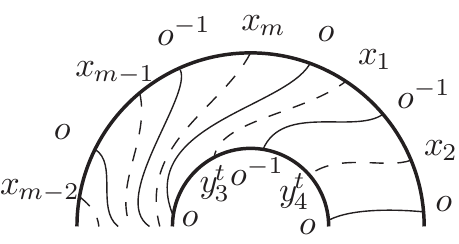}\hfill\hbox{}
\caption{\label{figure:paritypreserved} When $\epsilon_i =
  (-1)^i$ the only spoke diagrams that make a contribution
  are those where we connect an $o$ to an $o^{-1}$. This
  means we can only connect an $a$ to a $b$ if the indices
  have the same parity. This is what we see in equation
  (\ref{equation:paritypreserved}).}
\end{figure}
\end{proof}

\begin{definition}\label{definition:orthogonallyinvariant}
A random matrix is said to be invariant under conjugation by
an orthogonal matrix if the joint distribution of the
entries is invariant under conjugation by an orthogonal
matrix. So if we let $A$ be a random matrix, $O$ be a
orthogonal matrix and $B = OAO^{-1}$ then we mean that for
every $i_1, \dots, i_n, i_{-1}, \dots, i_{-n}$ we have
\[
\E(a_{i_1i_{-1}}\cdots a_{i_ni_{-n}}) =
\E(b_{i_1i_{-1}}\cdots b_{i_ni_{-n}}).
\]
\end{definition}

Many standard examples of random matrices are invariant
under conjugation by a unitary or orthogonal matrix. In
particular, real Wishart matrices, the Gaussian orthogonal
ensemble, Ginibre matrices, and orthogonal matrices are all
invariant under conjugation by an orthogonal matrix. In
\cite{r1,r2}, Redelmeier these were shown to have real second
order limit distributions and so satisfy the hypothesis of
our  theorem below.

\begin{theorem}\label{theorem:mastertheorem}
Suppose that $ \{ A_i \}_i $ and $ \{ B_j \}_j $ are two
independent families of random matrices, each with real
second order limit distribution. Suppose also that the
family $ \{A_i \}_i $ is invariant under conjugation by an
orthogonal matrix.  Then $ \{ A_i \}_i $ and $ \{ B_j \}_j $
are asymptotically real second order free.
\end{theorem}

\begin{proof}
Since the joint distribution of the entries of $A_i$ and
$OA_iO^{-1}$ are the same we may replace $\{A_i\}_i$ by
$\{OA_iO^{-1}\}_i$ and then apply
Proposition~\ref{prop:c02}.
\end{proof}


\section{Concluding Remark}\label{section:finalremarks}

Let us consider $ \{ A_i \}_i $ and $ \{ B_l \}_l $ two
independent ensembles of random matrices, each with a real
second order limit distribution and suppose that the
ensemble $ \{ A_i \}_i $ is invariant under a conjugation by
a unitary matrix.  In \cite{mss} it is shown that $ \{ A_i
\}_i $ and $ \{ B_l \}_l $ are asymptotically \emph{complex}
second order free (see \cite{mss}, Corollary 3.16). Since
orthogonal matrices are also unitary, Theorem
\ref{theorem:mastertheorem} implies that $ \{ A_i \}_i $ and
$ \{ B_l \}_l $ are asymptotically \emph{both} real and
complex second order free. In particular, the second term on
the right-hand side of equation
(\ref{equation:secondorderdefinition}) must vanish. In
consequence, for $ A_1, A_2 \in \{ A_i \}_i $ we have that
\[ \lim_{ d \rightarrow \infty } \text{tr}( A_1 A_2^t) = 0.\]
The connection between ensembles of random matrices which
are invariant under a conjugation with a unitary and real
second order freeness goes deeper than this and is
investigated in the subsequent paper \cite{mp2} in which we show that unitarily invariant ensembles are asymptotically free from their transposes.
 

\thebottomline
\end{document}